\documentclass[12pt,a4paper]{amsart}

\makeatletter
\long\def\unmarkedfootnote#1{{\long\def\@makefntext##1{##1}\footnotetext{#1}}}
\makeatother

\usepackage{pgf,tikz,pgfplots}
\pgfplotsset{compat=1.15}
\usepackage{mathrsfs}
\usetikzlibrary{arrows}

\usepackage{amsmath,amssymb}
\usepackage{enumerate,amssymb}
\usepackage{color}
\theoremstyle{plain}
\newtheorem{thm}{Theorem}[section]

\newtheorem{lemma}[thm]{Lemma}

\newtheorem{prop}[thm]{Proposition}
\newtoks\prt
\newtheorem{proclaim}[thm]{\the\prt}
\theoremstyle{definition}

\def\eqn#1$$#2$${\begin{equation}\label#1#2\end{equation}}

\numberwithin{equation}{section}

\def\C{\mathcal{C}}

\def\dist{\operatorname{dist}}

\def\epsilon{\varepsilon}
\def\en{\mathbb N}
\def\er{\mathbb R}

\def\haus{\mathcal{H}}

\def\loc{\operatorname{loc}}

\def\mir1{\mathcal L_1}

\def\oint{-\hskip -13pt \int}

\def\phi{\varphi}

\def\r{\widetilde{r}}

\def\sgn{\operatorname{sgn}}

\def\rn{\mathbb R^n}

\def\sgn{\operatorname{sgn}}

\def\V{\mathbb V}
\def\x{\widetilde{x}}

\newtoks\by
\newtoks\paper
\newtoks\book
\newtoks\jour
\newtoks\yr
\newtoks\pages
\newtoks\vol
\newtoks\publ
\def\ota{{\hbox\vol{???}}}
\def\cLear{\by=\ota\paper=\ota\book=\ota\jour=\ota\yr=\ota
\pages=\ota\vol=\ota\publ=\ota}
\def\endpaper{\the\by, {\the\paper},
\textit{\the\jour} \textbf{\the\vol} (\the\yr), \the\pages.\cLear}
\def\endbook{\the\by, \textit{\the\book}, \the\publ.\cLear}
\def\endprep{\the\by, \textit{\the\paper}, \the\jour.\cLear}
\def\endyearprep{\the\by, \textit{\the\paper}, \the\jour, (\the\yr).\cLear}
\def\name#1#2{#2 #1}
\def\nom{ \rm no. }

\newcommand{\seb}[1]{{ #1}}

\title{Injectivity in second-gradient Nonlinear Elasticity}

\author{D.~Campbell, S.~Hencl, A.~Menovschikov and S.~Schwarzacher}
\thanks{The first and the third author were supported by the grant GA\v{C}R 20-19018Y. The second author was supported by the grant GA\v{C}R P201/21-01976S. The fourth author was supported by the grant PRIMUS/19/SCI/01, the program \seb{GA\v{C}R GJ19-11707Y and by the ERC-CZ grant LL2105}}
\address{Department of Mathematics, University of Hradec Kr\'alov\'e, Rokitansk\'eho 62, 500 03 Hradec Kr\'alov\'e, Czech Republic}
\address{Faculty of Economics, University of South Bohemia, Studentsk\' a 13, Cesk\' e Budejovice, Czech Republic}
\email{\tt campbell@karlin.mff.cuni.cz, menovschikovmath@gmail.com}
\address{Department of Mathematical Analysis, Charles University, So\-ko\-lovsk\'a 83, 186~00 Prague 8, Czech Republic}
\email{\tt hencl@karlin.mff.cuni.cz, schwarz@karlin.mff.cuni.cz}

\date{\today}

\begin{document}

\begin{abstract}
We study injectivity for models of Nonlinear Elasticity that involve the second gradient. We assume that $\Omega\subset\rn$ is a domain, $f\in W^{2,q}(\Omega,\rn)$ satisfies 
$|J_f|^{-a}\in L^1$ and that $f$ equals a given homeomorphism on $\partial \Omega$. Under suitable conditions on $q$ and $a$ we show that $f$ must be a homeomorphism. 
As a main new tool we find an optimal condition for $a$ and $q$ that imply that $\haus^{n-1}(\{J_f=0\})=0$ and hence $J_f$ cannot change sign. \seb{We further specify in dependence of $q$ and $a$ the maximal Hausdorff dimension $d$ of the critical set $\{J_f=0\}$.
The sharpness of our conditions for $d$ is demonstrated by constructing respective counterexamples.} 
\end{abstract}

\maketitle

\section{Introduction}

	Our aim is to study models of Nonlinear Elasticity that involve second gradient and in particular we would like to study injectivity of such mappings. 
	Let $\Omega\subset\rn$ be a domain and let $f:\Omega\to\rn$. 
In this paper we study mappings with finite energy 
$$
E(f):=
\int_{\Omega} \bigl(W(Df(x))+\Psi(D^2 f(x))\bigr)\; dx
$$
such that there are $q\geq 1$ and $a>0$ so that  
\eqn{assume}
$$
W(Df)\geq \frac{1}{|J_f|^a}\text{ and }\Psi(D^2 f)\geq |D^2f|^q, 
$$
i.e. $f\in W^{2,q}$ and $J_f^{-a}\in L^1$. Moreover, we assume that the mapping $f$ is equal to a given homeomorphism $f_0$ on $\partial \Omega$. 
	 Models with the second gradient were introduced by Toupin \cite{T}, \cite{T2}
	and later considered by many other authors, see e.g. Ball, Curie, Olver \cite{BCO}, Ball, Mora-Corral \cite{BMC}, M\"uller \cite[Section 6]{Mbook}, Ciarlet \cite[page 93]{Ci} and references given there. 
	The contribution of the higher gradient is usually connected with interfacial energies and is used to model various phenomena like elastoplasticity or damage. \seb{Most relevant of \eqref{assume} is the penalisation of compression (by the negative powers of the Jacobian) that makes such energies a physical set up for the so-called regime of {\em large deformations}. Examples of recent usage are their connection to small strain-models~\cite{FriKru}, to thermovsicoelasticity~\cite{MiRo} and for the derivation of unsteady models involving inertia and interactions with fluids~\cite{BenKamSch}. For more references see the recent monograph~\cite{KrRo}. }

Injectivity in models of Nonlinear Elasticity is a crucial property as it corresponds to the ``non-interpenetration of matter". This question has attracted a huge amount of attention in the past for models with first order gradient. Following the pioneering work of J. Ball \cite{Ball} we can ask that our mapping has finite energy where the energy functional $\int_{\Omega} W(Df)$ contains special terms (like ratio of powers of $Df$, $\operatorname{adj} Df$ and $J_f$)
and any mapping with finite energy and reasonable boundary data is a homeomorphism (the reader is referred to e.g.\
\cite{HR,IS,MaVi} and \cite{Sv} for related results). There are other possible approaches that give us only some injectivity a.e.\ like Ciarlet and Ne\v{c}as condition \cite{CN} (see e.g. see e.g.~\cite{BK,BHM,GKMS,MTY,T}) or the (INV) condition of  M\"uller and Spector \cite{MS} (see e.g.\  \cite{BHM,ConDeL2003,HMC,MST,SwaZie2002,SwaZie2004, DHM}).

It is possible to show that for reasonable values of $q$ and $a$ (see \eqref{assume}) in second order models we obtain that the mapping with finite energy is indeed a homeomorphism. 
As far as we know the only result in this direction is the following result of T.J. Healey and Kr\"omer \cite{HK} (see also \cite{PH}): 

\prt{Theorem}
\begin{proclaim}
Let $q>n$ and let $(1-\frac{n}{q})a\geq n$. Let $\Omega\subset\rn$ be a domain and let $f_0:\overline{\Omega}\to\rn$ be a given homeomorphism. 
Assume that $f\in W^{2,q}(\Omega,\rn)$ is a mapping such that $|J_f|^{-a}\in L^1(\Omega)$ and $f=f_0$ on $\partial \Omega$ in the sense of traces. 
Then $f$ is a homeomorphism. 
\end{proclaim}

The main ingredient of their proof is the following: Since $f\in W^{2,q}$ with $q>n$ we obtain that $Df\in C^{0,1-\frac{n}{q}}$ and thus $J_f\in C^{0,1-\frac{n}{q}}$. 
We claim that the set $\{J_f=0\}$ is empty and thus $f$ is locally a homeomorphism and since it agrees with homeomorphism on the boundary it is a global homeomorphism. Assume for contrary that there is $x_0\in\Omega$ with $J_f(x_0)=0$, by using H\"older continuity of $J_f$ we get
$$
\int_{\Omega}\frac{1}{|J_f|^a}\geq C\int_{\Omega}\frac{1}{|x-x_0|^{a(1-\frac{n}{q})}}
$$ 
and the last integral is infinite since $a(1-\frac{n}{q})\geq n$. 

The above result is very important as a first breakthrough result in this direction. However, the condition $(1-\frac{n}{q})a\geq n$ is \seb{rather restrictive; in particular, the relevant case $q=2$ is excluded even in the planar case $n=2$}. Our first result \seb{extends and refines these results} and allows us to go with $q$ even below the critical case $q=n$. 

\prt{Theorem}
\begin{proclaim}
\label{main}
For $n\geq 3$ we assume that $a>n-1$ for $q>n$ and that $(1-\frac{n}{q}+\frac{1}{n-1})a>1$ for $n-1<q\leq n$. For $n=2$ we assume that $q>\frac{4}{3}$, $a\geq 1$ and $(\frac{3}{2}-\frac{2}{q})a\geq 1$. 
Let $\Omega\subset\rn$ be a Lipschitz domain and let $f_0:\overline{\Omega}\to\rn$ be a given homeomorphism. 
Assume that $f\in W^{2,q}(\Omega,\rn)$ is a mapping such that $|J_f|^{-a}\in L^1(\Omega)$ and $f=f_0$ on $\partial \Omega$ in the sense of traces. 
Then $f$ is a homeomorphism in $\Omega$. 
\end{proclaim}

Unlike the result of Healey and Kr\"omer we cannot show that the \seb{critical} set $\{J_f=0\}$ is empty, \seb{that is a popular property of the case treated there. Nevertheless, as we show below, the we are able to provide a rather direct excess to the critical set $\{J_f=0\}$.} The main new ingredient of our approach is to show that the critical set $\{J_f=0\}$ has zero Hausdorff $\haus^{n-1}$ measure. This is achieved by the following theorem, which \seb{actually provides a precise control on the size of the critical set.}

Let us note that for $f\in W^{2,q}$, $q>n$, we know \seb{(by Sobolev embedding)} that $Df$ is continuous and thus we take a continuous representative of $J_f$, \seb{which allows to make pointwise references to values where} $J_f(x)=0$. 
This is a little bit more \seb{delicate} for $q\leq n$ as $Df$ and $J_f$ \seb{are} defined only a.e.. But we know that $Df\in W^{1,q}$ and hence it has a quasicontinuous representative which is well-defined and has Lebesgue points up to a set of $q$-capacity zero or Hausdorff dimension $n-q$ (see Theorem \ref{qc} below). Therefore our $J_f$ is well-defined up to a set of dimension $n-q$. More precisely our assumption 
($(n-d+\frac{d}{n}-\frac{n}{q}(n-d))>0$ for $q\leq n$) even implies that $J_f$ has Lebesgue points and thus is well-defined up to a set of $\haus^d$ measure zero.  


\seb{The main result of the present paper is the following sharp specification of the Hausdorff dimension of the critical set of functions for which $E(f)<\infty$.}
\prt{Theorem}
\begin{proclaim}\label{technical}
Let $\Omega\subset\er^n$ be an open set, $q>n$, $0<d<n$ and let $(1-\frac{n-d}{q})a\geq n-d$. 
Assume that $f\in W^{2,q}(\Omega,\er^n)$ is a mapping with $|J_f|^{-a}\in L^1(\Omega)$. Then $\haus^d(\{J_f=0\})=0$. 
For $\frac{n^2}{2n-1}<q\leq n$ we assume that 
$(n-d+\frac{d}{n}-\frac{n}{q}(n-d)) a\geq n-d$ and we obtain that 
$\haus^{d}(\{J_f=0\})=0$.

Moreover, for every 
$0<d<n$ such that $(1-\frac{n-d}{q})a<n-d$ there is $C^1$ homeomorphism $f \in W^{2,q}((-1,1)^n, (-1,1)^n)$ with $J^{-a}_f \in L^1((-1,1)^n)$, such that $\mathcal{H}^d(\{J_f=0\})>0$. 
\end{proclaim}

The positive result for $q>n$ is sharp as the counterexample shows. However, there is a gap between the result and counterexample for $q\leq n$. 
We expect that the counterexample might be sharp also for $q\leq n$. 

As a corollary of \seb{the} previous theorem we obtain the following. 

\prt{Corollary}
\begin{proclaim}\label{cor}
Let $\Omega\subset\er^n$ be a domain, $q>n$ and $(1-\frac{1}{q})a\geq 1$ or $\frac{n^2}{2n-1}<q\leq n$ and $(2-\frac{1}{n}-\frac{n}{q}) a\geq 1$. 
Assume that $f\in W^{2,q}(\Omega,\er^n)$ is a mapping with $|J_f|^{-a}\in L^1(\Omega)$. Then either $J_f> 0$ a.e. in $\Omega$ or $J_f< 0$ a.e. in $\Omega$.
\end{proclaim}
\seb{The corollary shows} that $J_f$ does not change sign and hence we can assume that $f$ is a mapping of finite distortion (see e.g. \cite{HK14}). For suitable values of $a$ and $q$ we thus obtain that the distortion function $K_f(x)=\frac{|Df(x)|^n}{J_f(x)}$ is integrable with power $p>n-1$ (or $p\geq 1$ for $n=2$). 
Now we can use known results about mappings of finite distortion (see e.g. \cite{IS}, \cite{MaVi}, \cite{HR}, \cite{HK14} and \cite{Kr}) to conclude that $f$ is open and discrete and thus a homeomorphism and Theorem \ref{main} follows. 

Let us also note that it is possible to estimate the Hausdorff dimension of the image of the critical set (see e.g. Korobkov, Kristensen \cite{KK} and references given there). 

\seb{
The paper is structured as follows. In the next section we collect some preliminary results. This follows Section 3 where we prove the positive results. For that we first show that the zero set of {functions}, which are in $W^{2,b}$ and which have negative integrability has a maximal Hausdorff dimension. This follows the proofs of the positive part of Theorem~\ref{technical} and \ref{main}, as well as Corollary~\ref{cor}. In Section 4 the respective counterexamples are constructed. Both Section 3 and Section 4 rely on different known results but do involve new techniques and refinements. The paper is completed by two further sections. Section 5 where three more counterexamples  are introduced that relate to other known counter examples in the literature (e.g.~\cite{Ball}); in particular it includes the technical counterexample, where the set $\{J_f=0\}$ is shown to be dense. In the final section some positive implications of the developed theory are collected. In particular, some consequences of the {\em analytic results} of the paper to {\em minimizers} of $E$ are investigated.}

\seb{In conclusion it seems that the pure analytic possibilities to control the critical set of the Jacobian might be limited by the here provided results. However, it stays an untouched open problem, whether {\em minimizers} do enjoy better properties, or whether they can be as singular as the analytical counterexamples suggest.
}


\section{Preliminaries}

By $\haus^d$ we denote the $d$-dimensional Hausdorff measure in $\rn$. 

\subsection{Derivative of radial mapping}

\begin{lemma}\label{radial}
Let $\rho:(0,\infty)\to(0,\infty)$ be a strictly monotone function with $\rho\in C^1((0,\infty))$. Then, for the mapping
$$
h(x)=\frac{x}{|x|}\rho(|x|), \quad x\neq 0\ 
,$$
we have for every $x\neq 0$
\eqn{der}
$$
|Dh(x)|\seb{\approx}\max\Bigl\{\frac{\rho(|x|)}{|x|}, |\rho'(|x|)| \Bigr\}\text{ and }
 J_h(x)= \rho'(|x|) \Bigl(\frac{\rho(|x|)}{|x|}\Bigr)^{n-1}\ .
 $$
Moreover, if $\rho\in C^2((0,\infty))$, then
\eqn{secondder}
$$
|D^2h(x)|\seb{\approx} \max\Bigl\{|\rho''(|x|)|,\Bigl|\frac{\rho(|x|)}{|x|^2}-\frac{\rho'(|x|)}{|x|}\Bigr|\Bigr\}\text{ for every }x\neq 0. 
$$
\end{lemma}
\begin{proof}
The first part \eqref{der} follows from \cite[Lemma 2.1]{HK14}. To prove \eqref{secondder} we use a direct computation. 
The first order derivatives are 
\begin{align*}
\frac{\partial f_1}{\partial x_1}(x)   
& = \frac{\rho(|x|)}{|x|} + \frac{x_1}{|x|}\rho'(|x|)\frac{x_1}{|x|} - \frac{x_1}{|x|^2}\frac{x_1}{|x|}\rho(|x|)\\
\frac{\partial f_1}{\partial x_2}(x)  
& = \frac{x_1}{|x|}\rho'(|x|)\frac{x_2}{|x|} - \frac{x_1}{|x|^2}\frac{x_2}{|x|}\rho(|x|),
\end{align*}
and the second partial derivatives are
\begin{align*}
\frac{\partial^2 f_1}{\partial x_1^2}(x) 
& =  \frac{x_1\rho'(|x|)}{|x|^2} - \frac{x_1\rho(|x|)}{|x|^3} + \frac{2x_1\rho'(|x|)}{|x|^2} + \frac{x_1^3\rho''(|x|)}{|x|^3} - \frac{2x_1^3\rho'(|x|)}{|x|^4} - \frac{2 x_1\rho(|x|)}{|x|^3} \\
& - \frac{x_1^3\rho'(|x|)}{|x|^4} + \frac{3x_1^3\rho(|x|)}{|x|^5}, \\
\frac{\partial^2 f_1}{\partial x_2^2}(x) 
& =  \frac{x_1\rho'(|x|)}{|x|^2} + \frac{x_1x_2^2\rho''(|x|)}{|x|^3} - \frac{2x_1x_2^2\rho'(|x|)}{|x|^4} - \frac{x_1x_2^2\rho'(|x|)}{|x|^4} - \frac{x_1\rho(|x|)}{|x|^3} + \frac{3 x_1x_2^2\rho(|x|)}{|x|^5} ,\\
\frac{\partial^2 f_1}{\partial x_1 \partial x_2}(x) 
& =  \frac{x_2\rho'(|x|)}{|x|^2} - \frac{x_2\rho(|x|)}{|x|^3} + \frac{x_1^2x_2\rho''(|x|)}{|x|^3} - \frac{2x_1^2x_2\rho'(|x|)}{|x|^4} - \frac{x_1x_2\rho'(|x|)}{|x|^4} + \frac{3 x_1^2x_2\rho(|x|)}{|x|^5}, \\
\frac{\partial^2 f_1}{\partial x_2 \partial x_3}(x) 
& =  \frac{x_1 x_2 x_3\rho''(|x|)}{|x|^3} - \frac{2x_1x_2x_3\rho'(|x|)}{|x|^4} + \frac{3x_1x_2x_3\rho(|x|)}{|x|^5} - \frac{x_1x_2x_3\rho'(|x|)}{|x|^4}.
\end{align*}

By symmetry it is enough to calculate $|D^2f(x)|$ only for a single point on each sphere. Let us fix a point $y=[x_1, 0, \dots, 0]$ with $x_1>0$ and compute \seb{the} second derivatives at this point as
$$
\begin{aligned}
\frac{\partial^2 f_1}{\partial x_1^2}(y) & = \frac{\rho'(|x|)}{|x|} - \frac{\rho(|x|)}{|x|^2} + \frac{2\rho'(|x|)}{|x|} + \rho''(|x|) - \frac{2\rho'(|x|)}{|x|} - \frac{2\rho(|x|)}{|x|^2} - \frac{\rho'(|x|)}{|x|} + \frac{3\rho(|x|)}{|x|^2} \\
& = \rho''(|x|),\\
\frac{\partial^2 f_1}{\partial x_2^2}(y)& = \frac{\rho'(|x|)}{|x|} - \frac{\rho(|x|)}{|x|^2},\quad  
\frac{\partial^2 f_1}{\partial x_1 \partial x_2}(y) = 0,\quad  
\frac{\partial^2 f_1}{\partial x_2 \partial x_3}(y) = 0 .\\
\end{aligned}
$$
Other derivatives are computed similarly and hence \eqref{secondder} follows. 
\end{proof}

\subsection{Jacobian in $W^{1,b}$} 

\prt{Lemma}
\begin{proclaim}\label{Sebastian}
Let $\frac{n^2}{2n-1} < q\leq n$, and set $b=\frac{nq}{n^2-nq+q}$.   
Let $\Omega\subset \rn$ be open and $f\in W^{2,q}(\Omega,\rn)$. Then $J_f\in W^{1,b}(\Omega)$. 
\end{proclaim}
\begin{proof}
We know that $Df\in W^{1,q}$ and hence $Df$ satisfies the ACL condition (see e.g. \cite[Chapter 4.9]{EG}), i.e. all partial derivatives are absolutely continuous on almost all lines parallel to coordinate axes. Let us pick a segment such that all partial derivatives are absolutely continuous there. Then $J_f$ is a sum of products of absolutely continuous functions and hence it is absolutely continuous on this segment. 

To conclude that $J_f\in W^{1,b}$ it is thus enough to show that $D(J_f)\in L^b$ (see e.g. \cite[Chapter 4.9]{EG}). Clearly 
$$
|D(J_f)|\leq C |D^2f|\cdot |Df|^{n-1}
$$
and thus using H\"older's inequality with $p=\frac{q}{b}$ (note that $q>b$) we have
\eqn{kuk2}
$$
\int_{\Omega} |D(J_f)|^b\leq C\Bigl(\int_{\Omega} |D^2 f|^q\Bigr)^{\frac{1}{p}}
\Bigl(\int_{\Omega} |Df|^{(n-1)b\frac{p}{p-1}}\Bigr)^{\frac{p-1}{p}}.
$$
It is easy to check that
$$
b=\frac{nq}{n^2-nq+q}\text{ implies that }(n-1)b\frac{p}{p-1}=q^*=\frac{nq}{n-q}
$$
so that the last integral is finite and that $\frac{n^2}{2n-1}<q$ implies $b>1$. 
\end{proof}

\subsection{Quasicontinuous representative} Let us recall that each Sobolev mapping has a nice quasicontinuous representative (see \cite[Chapters 4.7 and 4.8]{EG}), i.e. a representative which has a Lebesgue points outside of a set of capacity zero. 

\prt{Theorem}
\begin{proclaim}\label{qc}
Let $\Omega\subset\rn$ be open, $1\leq b\leq n$ and let $g\in W^{1,b}(\Omega)$. Then there is a representative of $g$ and $E\subset \Omega$ with $\operatorname{Cap}_b(E)=0$ such that
$$
\text{ for every }x\in\Omega\setminus E\text{ we have }\lim_{r\to 0}\frac{1}{|B(x,r)|}\int_{B(x,r)}|g(y)-g(x)|\; dy=0.
$$
Moreover, the Hausdorff dimension of $E$ is equal to $n-b$. 
\end{proclaim}

\subsection{Poincar\'e inequality for functions that vanish on a set of positive Bessel capacity} 

We need the following version of Poincar\'e inequality for functions that vanish on a set of positive Bessel $b$-capacity (or Hausdorff dimension bigger than $n-b$) from \cite[Chapter 4.5. and Theorem 2.6.16]{Z}. 

\prt{Theorem}
\begin{proclaim}\label{Poincare}
Let $\Omega\subset\rn$ be open, $1\leq b<\infty$ and let $g\in W^{1,b}(\Omega)$. Then
\eqn{Poinc}
$$
\oint_B |g|\, dx\leq C r\oint_B|Dg|\, dx\leq Cr\bigg(\oint_B|Dg|^b\, dx\bigg)^{\frac{1}{b}}
$$
for each ball $B\subset\Omega$ of radius $r$ such that 
the Hausdorff dimension of $\{x\in B:\ g(x)=0\}$ is bigger than $n-b$ for $b\leq n$ and such that $\{x\in B:\ g(x)=0\}$ is nonempty for $b>n$. 
\end{proclaim}

\subsection{Mapping of finite distortion and injectivity} 

Let $\Omega\subset\rn$ be a domain. We say that a mapping $f\in W^{1,1}_{\loc}(\Omega,\rn)$ is a mapping of finite distortion if $J_f\in L^1_{\loc}(\Omega)$, $J_f\geq 0$ a.e. and 
$|Df(x)|$ vanishes a.e. in the set $\{x\in\Omega:\ J_f(x)=0\}$. For a mapping of finite distortion we define its distortion function as 
$$
K_f(x):=\begin{cases}
\frac{|Df(x)|^n}{J_f(x)}&\text{ if }J_f(x)\neq 0,\\
1&\text{ if }J_f(x)=0.\\
\end{cases}
$$
It is clear that each mapping with $J_f\in L^1_{\loc}(\Omega)$ and $J_f>0$ a.e.\ is a mapping of finite distortion. We need the following result about injectivity of mappings of finite distortion.

\prt{Theorem}
\begin{proclaim}\label{opendiscr}
Let $\Omega\subset\rn$ be a Lipschitz domain and let $f_0:\overline{\Omega}\to\rn$ be a homeomorphism. Assume that $f\in W^{1,1}_{\loc}(\Omega,\rn)$ is a mapping of finite distortion such that $f\in C(\overline{\Omega},\rn)$, $f=f_0$ on $\partial\Omega$ in the sense of traces, $K_f\in L^1(\Omega)$ for $n=2$ and $K_f\in L^p(\Omega)$ for some $p>n-1$ for $n\geq 3$. 
Then $f$ is a homeomorphism on $\Omega$. 
\end{proclaim}
\begin{proof}
This essentially follows from \cite[Theorem 6.8]{Kr}. The only thing we need to verify is that $\deg(f,\Omega,z)\leq 1$ for every $z\in\rn\setminus f(\partial\Omega)$ (i.e. that $f\in DEG1$ class from \cite{Kr}). However $f$ is continuous up to the boundary and is equal to a homeomorphism $f_0$ on $\partial\Omega$ and thus $\deg(f,\Omega,z)=\deg(f_0,\Omega,z)$. Now each homeomorphism has degree either $1$ or $-1$ on $f_0(\Omega)$ so we can assume without loss of generality that it is $1$. 
\end{proof}

\section{Proof of positive results}

\prt{Theorem}
\begin{proclaim}\label{qqq}
Let   $\Omega\subset\er^n$ be an open set, $0<d<n$, $1\leq b<\infty$, $a>0$ and $(1-\frac{n-d}{b})a\geq n-d$. 
Assume that $g\in W^{1,b}(\Omega)$ is a mapping with $|g|^{-a}\in L^1(\Omega)$. Then $\haus^d(\{g=0\})=0$. 
\end{proclaim}

\begin{proof}
Assume by contradiction that $\haus^d(\{g=0\})=c_0>0$. Given $\epsilon>0$
 we can find disjoint balls $B_i=B(c_i,r_i)\subset\Omega$ with $\haus^d(\{g=0\}\cap B_i)>0$ such that $\sup_i r_i<\tilde{\epsilon}$ and $C c_0\leq \sum_{i=1}^\infty r^d_i $. Further by the fact that $|g|^{-a}\in L^1$, we find that $\haus^n(\{g=0\})=0$. This implies that the union of these balls can be assumed to cover a small enough measure such that by the absolute continuity of the integral 
\eqn{abscont}
$$
\int_{\bigcup_i B_i}|g|^{-a}\; dx<\epsilon.
$$ 

Note that $(1-\frac{n-d}{b})a\geq n-d$ implies that $d>n-b$. Therefore $\haus^d(\{g=0\}\cap B_i)>0$ implies that we can use \eqref{Poinc} to obtain
\eqn{poinc}
$$
\oint_{B_i} |g|\, dx\leq Cr_i\bigg(\oint_{B_i}|Dg|^b\, dx\bigg)^{\frac{1}{b}}. 
$$
%
H\"older's inequality gives us 
\begin{align*}
Cr_i^n&\leq \int_{B_i} |g|^{\frac{a}{a+1}}\frac{1}{|g|^{\frac{a}{a+1}}}\leq C\Bigl(\int_{B_i} |g|\Bigr)^{\frac{a}{a+1}}\Bigl(\int_{B_i} \frac{1}{|g|^a}\Bigr)^{\frac{1}{a+1}}.
\end{align*}
After we raise this to power $a+1$ and use \eqref{poinc} we get 
\begin{align*}
Cr_i^{n+an}
\leq C\Bigl( r_i^{1+n-\frac{n}{b}}\Bigl(\int_{B_i}|Dg|^b\Bigr)^{\frac{1}{b}}\Bigr)^{a}\int_{B_i} \frac{1}{|g|^a}\, dx,
\end{align*}
which leads to the following key estimate
\eqn{firsttry}
$$
\frac{r_i^{n-a(1-\frac{n}{b})}}{\Bigl(\int_{B_i}|Dg|^b\Bigr)^{\frac{a}{b}}}
\leq C \int_{B_i} \frac{1}{|g|^a} 
$$
and therefore 
$$
\epsilon>\sum_i \int_{B_i}\frac{1}{|g|^a}\geq \sum_i \frac{r_i^{n-a(1-\frac{n}{b})}}{\Bigl(\int_{B_i}|Dg|^b\Bigr)^{\frac{a}{b}}}. 
$$
Now we use H\"older's inequality 
$$
\Bigl(\sum_{i=1}^\infty a_i^{\alpha}\Bigr)^{\frac{1}{\alpha}} \geq \frac{1}{\Bigl(\sum_{i=1}^\infty b_i^{\beta} \Bigr)^{\frac{1}{\beta}}} \sum_{i=1}^\infty a_i b_i
$$
where we choose $\frac{\beta}{\alpha}=\frac{b}{a}$ which means
for $\alpha = \frac{a+b}{b}$, $\beta=\frac{a+b}{a}$,
\[b_i = \bigg(\int_{B_i} |Dg(x)|^b dx \bigg)^{\frac{a}{b\alpha}}\text{ and }a_i = \bigg(\frac{r_i^{n-a(1-\frac{n}{b})}}{\big(\int_{B_i}|Dg|^b\, dx\big)^{\frac{a}{b}}}\bigg)^{\frac{1}{\alpha}}
\]
 and we infer
\begin{align*}
\epsilon^\frac{1}{\alpha} > C \frac{\sum_{i=1}^\infty r_i^\frac{n-a(1-\frac{n}{b})}{\alpha}}{\big(\sum_{i=1}^\infty \int_{B_i}|Dg|^b\, dx\big)^\frac{a}{b}}.
\end{align*}
Since $\int_{\Omega}|Dg|^b\leq C$ this implies that
\[
C\epsilon^\frac1{\alpha}>\sum_{i=1}^\infty r_i^\frac{b(n-a(1-\frac{n}{b}))}{a+b}
\] 
and our condition $(1-\frac{n-d}{b})a\geq n-d$ implies that 
\[
d\geq \frac{b(n-a(1-\frac{n}{b}))}{a+b}
\]
giving us
$$
C\epsilon^\frac1{\alpha}>\sum_{i=1}^\infty r_i^d\geq C c_0
$$
leading to a desired contradiction. 
\end{proof}

\begin{proof}[Proof of positive results in Theorem \ref{technical}]
Let us first assume that $q>n$. As $f\in W^{2,q}(\Omega,\er^n)$, $q>n$, we can use Sobolev embedding and obtain $Df\in C^{0,1-\frac{n}{q}}$. It follows easily that 
$|Df|\in L^{\infty}(\Omega)$ and it is not difficult to conclude that $J_f\in W^{1,q}(\Omega)$. Our assumption $(1-\frac{n-d}{q})a\geq n-d$ implies that we can use 
Theorem \ref{qqq} for $b=q$ and $g=J_f$ and we obtain our conclusion. 

Let us now assume that $\frac{n^2}{2n-1}<q\leq n$. From Lemma \ref{Sebastian} we obtain that $J_f\in W^{1,b}$ for $b=\frac{nq}{n^2-nq+q}$. It is easy to see that 
$$
\Bigl(1-\frac{n-d}{b}\Bigr)a=\Bigl(n-d+\frac{d}{n}-\frac{n}{q}(n-d)\Bigr) a\geq n-d
$$
where we have used our assumption in the last inequality. The conclusion now follows from Theorem \ref{qqq}. 
\end{proof}

\begin{proof}[Proof of Corollary \ref{cor}]
Let us denote by $P_i$ the projection to the hyperplane $\{x_i=0\}$, $i\in\{1,\hdots,n\}$. 
Using Theorem \ref{technical} for $d=n-1$ we obtain that 
$\mathcal{H}^{n-1}(\{J_f=0\})=0$ and hence also $\mathcal{H}^{n-1}(P_i(\{J_f=0\}))=0$. Here we use a continuous representative of $J_f$ for $q>n$ 
and quasicontinuous representative from Theorem \ref{qc} for $q\leq n$ which is correctly defined at Lebesgue points, i.e. up to a set of dimension $n-b=n-\frac{nq}{n^2-nq+q}<d=n-1$. 
It is well-known that this quasicontinuous representative satisfies the ACL condition (see e.g. \cite[Chapter 4.9]{EG}) so especially  it is continuous on lines
$\Omega\cap P_i^{-1}(\{a\})$ for $\haus^{n-1}$-a.e. $a\in\er^{n-1}$. 


Assume for contrary both $\{J_f>0\}$ and $\{J_f<0\}$ have positive measure. Then we can find a direction $i\in\{1,\hdots,n\}$ such that 
there is $A\subset\er^{n-1}$ with $\mathcal{H}^{n-1}(A)>0$ such that for every $a\in A$ we know that $J_f$ is continuous on $\Omega\cap P_i^{-1}(\{a\})$ and 
$\Omega\cap P_i^{-1}(\{a\})$ contains a segment $L$ with 
$$
L\cap\{J_f>0\}\neq\emptyset\text{ and }L\cap\{J_f<0\}\neq\emptyset. 
$$
It follows that there is $x_a\in L\subset \Omega\cap P_i^{-1}(\{a\})$ with $J_f(x_a)=0$ which gives us a contradiction. 
\end{proof}

\begin{proof}[Proof of positive results in Theorem \ref{main}]
We claim that $f$ satisfies the assumptions of Corollary \ref{cor}. Indeed, for $n=2$ we use the same assumptions for $q$ and $a$ and for $n\geq 3$ we have 
$$
\Bigl(1-\frac{1}{q}\Bigr)a\geq \Bigl(1-\frac{n}{q}+\frac{1}{n-1}\Bigr)a>1\text{ for }n-1<q\leq n
$$
$$
\text{ and }a>n-1\text{ implies that }\Bigl(1-\frac{1}{q}\Bigr)a\geq 1\text{ for }q>n. 
$$
Hence we can assume without loss of generality that $J_f>0$ a.e. 
Since $f\in W^{2,q}$, $q>n-1$, implies that $Df\in W^{1,n}$ and thus $J_f\in L^1$ we obtain that $f$ is a mapping of finite distortion.

To get our conclusion we apply Theorem \ref{opendiscr}. We know that $f\in W^{2,q}$ for $q>n-1$ and hence $f$ is continuous and since $\Omega$ is Lipschitz it is standard that even $f\in C(\overline{\Omega},\rn)$. It remains to show that $K_f\in L^1$ for $n=2$ and that $K_f\in L^p$ for some $p>n-1$ for $n\geq 3$. 

Let us first consider $n=2$ and $K_f=\frac{|Df|^2}{J_f}\in L^1$. 
For $q>2$ we obtain that $|Df|\in L^{\infty}$ and the condition $a\geq 1$ implies that $K_f\in L^1$. 
In the case $\frac{4}{3}<q<2$ we use H\"older's inequality 
$$
\int_{\Omega} \frac{|Df|^2}{J_f}\leq \Bigl(\int_{\Omega} \frac{1}{J_f^a}\Bigr)^{\frac{1}{a}}\Bigl(\int_{\Omega} |Df|^{2\frac{a}{a-1}}\Bigr)^{\frac{a-1}{a}}.
$$
Clearly $(2-\frac{2}{q})a\geq (\frac{3}{2}-\frac{2}{q})a\geq 1$ and thus $(2-\frac{2}{q})a\geq 1$ implies that
$$
2\frac{a}{a-1}\leq q^*=\frac{2q}{2-q}
$$
and the last integral is finite by $Df\in L^{q*}$. Similarly for $q=2$ we use H\"older's inequality, $a\geq 2$ and $Df\in L^4$. 

It remains to consider $n\geq 3$ and $K_f=\frac{|Df|^n}{J_f}\in L^p$ for some $p>n-1$. For $q>n$ we obtain that $|Df|\in L^{\infty}$ and the condition $a>n-1$ gives us our conclusion. 
For $n-1<q<n$ we obtain 
$$
\int\frac{|Df|^{np}}{J_f^p}\leq \Bigl(\int \frac{1}{J_f^a}\Bigr)^{\frac{p}{a}}\Bigl(\int |Df|^{np\frac{a}{a-p}}\Bigr)^{1-\frac{p}{a}}. 
$$
Our assumption 
$$
\frac{n(q-n+1)}{q(n-1)}a=\Bigl(1-\frac{n}{q}+\frac{1}{n-1}\Bigr)a>1
$$
implies that for $p$ sufficiently close to $n-1$ we obtain 
$$
np\frac{a}{a-p}< q^*=\frac{nq}{n-q}
$$
and the last integral is finite by $Df\in L^{q^*}$. The case $q=n$ is again similar as 
$|Df|$ is integrable with any power and $a>n-1$. 
\end{proof}

\section{Counterexample for the size of the Hausdorff measure of $\{J_f=0\}$}\label{ahoj}

The following counterexample is a generalization of the Cantor type construction of Ponomarev \cite{P} (see also \cite[Section 4.3]{HK14}) for mappings with higher order derivatives, which was provided by Roskovec \cite{R18}. In fact we need a generalization of this construction since we map squares to rectangles instead of squares to squares as in \cite{HKM}. 

\begin{proof}[Proof of counterexample in Theorem \ref{technical}]
Given that $(1-\frac{n-d}{q})a<n-d$, we would like to construct a homeomorphism $f \in W^{2,q}((-1,1)^n, (-1,1)^n)$ with $J^{-a}_f \in L^1((-1,1)^n)$, such that $\mathcal{H}^d(\{J_f=0\})>0$. 
First we construct a Cantor set $\C_Q$ with $\mathcal{H}^d(\C_Q)>0$ and a Cantor type set $\C_R$ and then we construct $f$ which maps $\C_Q$ onto $\C_R$ and satisfies $J_f=0$ on $\C_Q$.

{\underline{Step 1: Construction of Cantor sets $\C_Q$ onto $\C_R$:}}	To start, let us determine an important parameter $\beta$. The condition $(1-\frac{n-d}{q})a<n-d$ implies $\frac{n}{d}\frac{q-n+d}{q}<\frac{n}{d}\frac{n-d}{a}$. We call $\beta$ the midpoint between the two, i.e.
	\begin{equation}\label{FixBeta}
		\frac{n}{d}\frac{q-n+d}{q}<\beta = \frac{n}{d}\frac{q-n+d}{2q} + \frac{n}{d}\frac{n-d}{2a}<\frac{n}{d}\frac{n-d}{a}.
	\end{equation}
	
	For $x = [x^1, \dots, x^n] \in \mathbb{R}^n$ denote
	$$
	\bar{x} = [x^1, \dots, x^{n-1}] \in \mathbb{R}^{n-1}.
	$$
	
	The first Cantor type set is the same, as in \cite{HK14}. Denote by $\mathbb{V}$ the set of vertices of the cube $[-1,1]^n$. Now $\mathbb{V}^i = \mathbb{V} \times \ldots \times \mathbb{V}$ serves as an index set for our construction. Let us set $z_0=\tilde z_0 = 0$ and denote
\eqn{defab}
	$$
		a_i = 2^{-\frac{n}{d}i}, \quad b_i = 2^{-(\frac{n}{d}+\beta)i},
	$$
	and we define cubes and rectangles 
	\begin{align*}
	& Q(x,r) = \{y \in \mathbb{R}^n : \|x-y\|_\infty \leq r\}, \\
	& R(x,l, w) = \{y \in \mathbb{R}^n:  \|\bar{x}-\bar{y}\|_\infty \leq l,  |x^n-y^n|\leq w\}.
	\end{align*}
	
	Then $[-1,1]^n = Q(z_0, a_0) = R(z_0, a_0, b_0)$. We proceed by induction. For $v = [v_1, \ldots, v_i] \in \mathbb{V}^i$ we denote $w = [v_1,\ldots,v_{i-1}]$ and define points
	\begin{align*}
	& z_v = z_w +\frac{1}{2}(a_{i-1})v_i = z_0 + \sum\limits_{j=1}\limits^i\frac{1}{2}(a_{j-1})v_j, \\
	\tilde{z}_v = [\tilde{z}^1_v, \dots, \tilde{z}^{n-1}_v, \tilde{z}^n_v] = &\Bigl[\tilde{z}^1_w +\frac{1}{2}(a_{i-1})v^1_i, \dots, \tilde{z}^{n-1}_w +\frac{1}{2}(a_{i-1})v^{n-1}_i , \tilde{z}^n_w +\frac{1}{2}(b_{i-1})v^n_i\Bigr],
	\end{align*}
	and we define squares and rectangles centered at these points (see Fig. \ref{Fig:Rectangles})
	\eqn{rectangle}
	$$
	\begin{aligned}
	& Q'_v = Q(z_v, \frac{1}{2}a_{i-1}),  \quad Q_v =Q(z_v, a_i), \\
	& R'_v = R(\tilde{z}_v, \frac{1}{2}a_{i-1}, \frac{1}{2}b_{i-1}), \quad R_v = R(\tilde{z}_v, a_i, b_i). 
	\end{aligned}
	$$
	
	\begin{figure}\unitlength=0.9mm
		\begin{picture}(110,50)(5,5)
		\put(10,10){\framebox(40,40){}}
		\put(30,10){\line(0,1){40}}
		\put(10,30){\line(1,0){40}}
		\put(12,16){\framebox(16,8){}}
		\put(32,16){\framebox(16,8){}}
		\put(12,36){\framebox(16,8){}}
		\put(32,36){\framebox(16,8){}}
		\put(60,10){\framebox(40,40){}}
		\put(80,10){\line(0,1){40}}
		\put(60,30){\line(1,0){40}}
		\put(62,16){\framebox(16,8){}}
		\put(70,16){\line(0,1){8}}
		\put(62,20){\line(1,0){16}}
		\put(63,17){\framebox(6,2){}}
		\put(63,21){\framebox(6,2){}}
		\put(71,17){\framebox(6,2){}}
		\put(71,21){\framebox(6,2){}}
		\put(62,36){\framebox(16,8){}}
		\put(70,36){\line(0,1){8}}
		\put(62,40){\line(1,0){16}}
		\put(63,37){\framebox(6,2){}}
		\put(63,41){\framebox(6,2){}}
		\put(71,37){\framebox(6,2){}}
		\put(71,41){\framebox(6,2){}}
		\put(82,16){\framebox(16,8){}}
		\put(90,16){\line(0,1){8}}
		\put(82,20){\line(1,0){16}}
		\put(83,17){\framebox(6,2){}}
		\put(83,21){\framebox(6,2){}}
		\put(91,17){\framebox(6,2){}}
		\put(91,21){\framebox(6,2){}}
		\put(82,36){\framebox(16,8){}}
		\put(90,36){\line(0,1){8}}
		\put(82,40){\line(1,0){16}}
		\put(83,37){\framebox(6,2){}}
		\put(83,41){\framebox(6,2){}}
		\put(91,37){\framebox(6,2){}}
		\put(91,41){\framebox(6,2){}}
		\end{picture}
		\caption{An illustration of the first and second generation rectangles $R_{v}$ and $R'_{v}$ of the construction of the Cantor set for $n=2$.}\label{Fig:Rectangles}
	\end{figure}
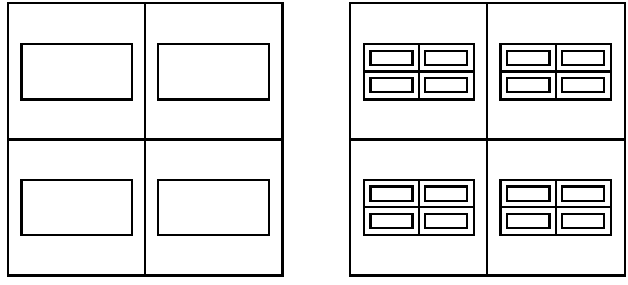

	The resulting Cantors sets 
	$$
	\C_Q=\bigcap_{k\in\en} \bigcup_{v\in \V^k} Q_{v} \text{ and }\C_R=\bigcap_{k\in\en} \bigcup_{v\in \V^k} R_{v}
	$$  
	are the products of $n$ Cantor sets in $\mathbb{R}$, with $\mathcal{L}_n(\C_Q) = \mathcal{L}_n(\C_R) = 0$. It is not difficult to find out that $\mathcal{H}^d(\C_Q)>0$ since we have 
	$2^{ni}$ cubes of sidelength $a_i=2\cdot 2^{-\frac{n}{d}i}$ in the $i$-th step of the construction. 
	Notation $\C_Q$ and $\C_R$ corresponds to the fact that $\C_Q$ is an intersection of cubes and $\C_R$ is an intersection of rectangles. 
	
	Our mapping $f$ is defined as the limit of the sequence of smooth homeomorphisms $f_i: (-1,1)^n \to (-1,1)^n$, where each $f_i$ maps the $i$-th generation cubes of the first Cantor set $\C_Q$ onto the $i$-th generation rectangles of the second set $\C_R$. In the last coordinate it is squeezing cube of size $a_i$ to rectangle of size $b_i<<a_i$ and hence in the 
	limit $J_f$ vanishes on $\C_Q$. 

{\underline{Step 2: Auxiliary function $h_i$:}}	
 Following the ideas of Roskovec \cite{R18}, we consider the convolution kernel $\phi: (-1,1) \to \mathbb{R}$ such that
	\begin{enumerate}
		\item $\phi(t) \geq 0$,
		\item $\phi(-t) = \phi(t)$,
		\item $\int^1_{-1}\phi(t)dt = 1$,
		\item $|D^2 \phi(t)| \leq C$,
		\item $\int^1_{-1}|D^2 \phi(t)| dt \leq C$,
		\item $\seb{\phi \in C^\infty_0}((-1,1))$.
	\end{enumerate}
	Then, for $r>0$, the function $\phi_r(t) = r^{-1}\phi(r^{-1}t)$ satisfies
	\begin{enumerate}
		\item $\phi_r(t) \geq 0$,
		\item $\phi(-t) = \phi(t)$,
		\item $\int^r_{-r}\phi_r(t)dt = 1$,
		\item $|D^2 \phi_r(t)| \leq r^{-3}C$,
		\item $\int^r_{-r}|D^2 \phi_r(t)| dt \leq r^{-2}C$,
		\item $\seb{\phi_r} \in C^\infty_0((-r,r))$.
	\end{enumerate}
	
	The next step is to define smooth mappings $g_i : Q(0, \frac{a_{i-1}}{2}) \to R(0, \frac{a_{i-1}}{2}, \frac{b_{i-1}}{2})$, from which we construct the mappings $f_i$. 
	Our construction guarantees that the mapping is identical in the first $n-1$ coordinates and is smooth in the last one. 
	To define the $n$-th component of $g_i$, for $i \in \mathbb{N}$, we start by defining a preliminary function $h^\ast_i(t): [0, \tfrac{1}{2}a_{i-1}] \to \mathbb{R}$ as follows:
	$$
	h^\ast_i(t)=
	\begin{cases}
	\frac{b_i}{a_i} t,& \quad t \in [0, a_i + \tfrac{1}{4}(\tfrac{1}{2}a_{i-1} - a_i)],\\
	l_i(t),& \quad t \in [a_i + \tfrac{1}{4}(\tfrac{1}{2}a_{i-1} - a_i), \tfrac{1}{2}a_{i-1} - \tfrac{1}{4}(\tfrac{1}{2}a_{i-1} - a_i)],\\
	\frac{b_{i-1}}{a_{i-1}} t,& \quad t \in [\tfrac{1}{2}a_{i-1} - \tfrac{1}{4}(\tfrac{1}{2}a_{i-1} - a_i), \tfrac{1}{2}a_{i-1}],
	\end{cases}
	$$
	where $l_i$ is the uniquely determined linear function so that $h^*_i$ is continuous on $[0,\tfrac{1}{2}a_{i-1}]$ (see Figure~\ref{Fig:Graph}). Notice that 
	$$
	0<\frac{b_i}{a_i} = 2^{-\beta}\frac{b_{i-1}}{a_{i-1}} < \frac{b_{i-1}}{a_{i-1}}\text{ and }
	a_i + \tfrac{1}{4}(\tfrac{1}{2}a_{i-1} - a_i)< \tfrac{1}{2}a_{i-1} - \tfrac{1}{4}(\tfrac{1}{2}a_{i-1} - a_i)
	$$ 
	and therefore
	$$
	\frac{b_i}{a_i}[a_i + \tfrac{1}{4}(\tfrac{1}{2}a_{i-1} - a_i)] < \frac{b_{i-1}}{a_{i-1}}[a_i + \tfrac{1}{4}(\tfrac{1}{2}a_{i-1} - a_i)] < \frac{b_{i-1}}{a_{i-1}}[\tfrac{1}{2}a_{i-1} - \tfrac{1}{4}(\tfrac{1}{2}a_{i-1} - a_i)].
	$$
	This means that $l_i$ is increasing and therefore $h^\ast_i$ is an increasing continuous piecewise linear function on $[0,\tfrac{1}{2}a_{i-1}]$.
	
	We define $h_i$ by smoothing $h^\ast_i(t)$ using $\phi_{r_i}(t)$ for $r_i =\tfrac{1}{16}(\tfrac{1}{2}a_{i-1} - a_i) $, specifically we have
	$$
	h_i(t) =\begin{cases}
	\frac{b_i}{a_i} t,& \quad t \in [0, a_i],\\
	\int_{-r_i}^{r_i} h^\ast_i(t+s) \phi_{r_i}(s)ds,& \quad t \in [a_i, \tfrac{1}{2}a_{i-1} - \tfrac{1}{8}(\tfrac{1}{2}a_{i-1} - a_i)],\\
	\frac{b_{i-1}}{a_{i-1}} t,& \quad t \in [\tfrac{1}{2}a_{i-1} - \tfrac{1}{8}(\tfrac{1}{2}a_{i-1} - a_i), \tfrac{1}{2}a_{i-1}].
	\end{cases} 
	$$
	We have that $h^*_i$ is linear on $[0,a_i + \tfrac{1}{4}(\tfrac{1}{2}a_{i-1} - a_i)]$. This fact and property (2) of $\phi_r$ easily yield that $h_i(t) = h_i^*(t)$ for $t \in [0,a_i + \tfrac{3}{16}(\tfrac{1}{2}a_{i-1} - a_i)]$. Similarly $h_i(t) = h_i^*(t)$ for $t \in [\tfrac{1}{2}a_{i-1} - \tfrac{3}{16}(\tfrac{1}{2}a_{i-1} - a_i), \tfrac{1}{2}a_{i-1}]$. Easily we get that $h_i$ is an increasing function.
	
	\begin{figure}
		\definecolor{ffqqqq}{rgb}{1.,0.,0.}
		\begin{tikzpicture}[line cap=round,line join=round,>=triangle 45,x=0.8cm,y=0.8cm]
		\clip(0.,-1.5) rectangle (13.,6.5);
		\fill[line width=2.pt,color=ffqqqq,fill=ffqqqq,fill opacity=0.10000000149011612] (8.75,9.) -- (8.75,-2.) -- (9.25,-2.) -- (9.25,9.) -- cycle;
		\fill[line width=2.pt,color=ffqqqq,fill=ffqqqq,fill opacity=0.10000000149011612] (10.75,9.) -- (10.75,-2.) -- (11.25,-2.) -- (11.25,9.) -- cycle;
		\draw [line width=0.8pt,domain=0.:13.] plot(\x,{(-0.-0.*\x)/16.});
		\draw [line width=0.8pt,dotted] (12.,-0.1) -- (12.,8.5);
		\draw [line width=0.8pt,dotted] (8.,-0.1) -- (8.,8.5);
		\draw [line width=0.8pt,dotted] (0.6,0.5) -- (13.,0.5);
		\draw [line width=0.8pt,dotted] (0.8,6) -- (13.,6);
		\draw [line width=0.8pt,dotted] (9.,-0.8) -- (9.,8.5);
		\draw [line width=0.8pt,dotted] (11.,-0.8) -- (11.,8.5);
		\draw [line width=0.8pt] (9.,0.5625)-- (11.,5.5);
		\draw [line width=0.8pt] (12.,6.)-- (11.,5.5);
		\draw [line width=0.8pt] (9.,0.5625)-- (0.,0.);
		\draw [line width=0.8pt,dotted] (11.,5.5)-- (0.,0.);
		\draw [line width=0.8pt,dotted] (9.,0.5625)-- (12.,0.75);
		\draw [line width=0.8pt,color=ffqqqq] (8.75,9.)-- (8.75,-0.5);
		\draw [line width=0.8pt,color=ffqqqq] (9.25,-0.5)-- (9.25,9.);
		\draw [line width=0.8pt,color=ffqqqq] (10.75,9.)-- (10.75,-0.5);
		\draw [line width=0.8pt,color=ffqqqq] (11.25,-0.5)-- (11.25,9.);
		\draw (0,0) node[anchor=north west] {$0$};
		\draw (7.7,0) node[anchor=north west] {$a_i$};
		\draw (0,0.8) node[anchor=north west] {$b_i$};
		\draw (0,6.3) node[anchor=north west] {$\tfrac{1}{2}b_{i-1}$};
		\draw (11.5,0) node[anchor=north west] {$\tfrac{1}{2}a_{i-1}$};
		\draw (7.4,-0.7) node[anchor=north west] {$\tfrac{1}{8}a_{i-1}+\tfrac{3}{4}a_{i} $};
		\draw (10.3,-0.7) node[anchor=north west] {$\tfrac{3}{8}a_{i-1} + \tfrac{1}{4}a_i$};
		\end{tikzpicture}
		\caption{A sketch of the graph of $h^*_i$. The interval $[a_i, \tfrac{1}{2}a_{i-1}]$ is separated into 3 parts with the central part twice as long as each of the ones on the sides. The smooth function, $h_i$, is equal to $h^*_i$ outside the two strips coloured in red. The width of the red strips is $\tfrac{1}{4}$ the length of the narrower parts.}\label{Fig:Graph}
	\end{figure}
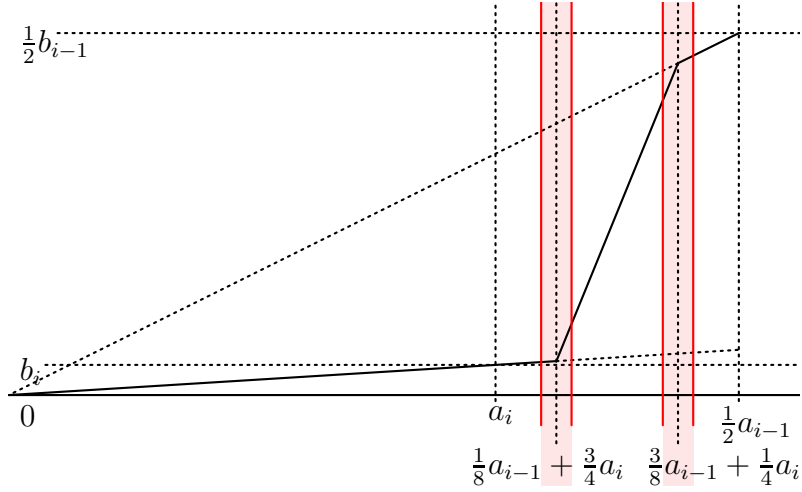
	
	We estimate (see Figure~\ref{Fig:Graph} and \eqref{defab}) that
	$$
		h_i'(t) \leq \max_{t\in [0,\tfrac{1}{2}a_{i-1}]}(h^*_{i})'(t) 
		\leq C(n,d,\beta)\frac{b_i}{a_i}.
	$$
	In fact we see from the definition of $h_i^*$ that
	\begin{equation}\label{hdif1}
		h_i'(t) \approx 
		\seb{\frac{b_i}{a_i}}.
	\end{equation}
	Further it is easy to estimate that $h_i'' \leq r_i^{-1}\|h'_i\|_{\infty, [0, \tfrac{1}{2}a_{i-1}]}$. Since
	$$
		r_i = \tfrac{1}{16}(\tfrac{1}{2}a_{i-1} - a_i) = 2^{-i\tfrac{n}{d} - 4}(2^{\tfrac{n}{d} -1} - 1) = C(n,d)a_i
	$$
	we have
	\begin{equation}\label{hdif2}
	|h''_i(t)| \leq C(n,d,\beta)\frac{b_i}{a_i^2}.
	\end{equation}

	Now we define a preliminary map $g^*_i : Q(0, \frac{a_{i-1}}{2}) \to R(0, \frac{a_{i-1}}{2}, \frac{b_{i-1}}{2})$ as
	$$
	g^*_i(x) = \bigl[\bar{x}, \sgn(x^n)h_i(|x^n|)\bigr].
	$$
	The shifted version of $\tilde{z}_v + g^*_i(\cdot - z_v)$ maps $Q'_v$ onto $R'_v$ for any $v \in \mathbb V^i$. This map is continuous, smooth and strictly monotone in every direction, and therefore it is a homeomorphism.

{\underline{Step 3: Modification of $h_i$ to $\overline{h}_i$:}}		
	Our final modification of $g^*_i$ before we construct $f: [-1,1]^n \to [-1,1]^n$, is intended to guarantee that the map smoothly coincides with the linear map that sends $Q'_v$ onto $R'_v$ on the boundary of $Q'_v$. By the construction of $h_i$, this is already true for the faces of $Q'_v$ where $x^n$ is extremal but now we are concerned with the faces where one of  $x^1, \dots, x^{n-1}$ is extremal. To do this, we consider the following smooth function $ \lambda_{a, b} : \er \to [0,1]$ for $0<a<b$ such that
	\begin{enumerate}
		\item $\lambda_{a, b}(t) \in C^\infty(\er)$,
		\item $\lambda_{a, b}(t) \equiv 1$ for $t \in (-\infty,a)$,
		\item $\lambda_{a, b}(t) \equiv 0$ for $t \in (b,\infty)$,
		\item $0<\lambda_{a,b}'(t) \leq C(b-a)^{-1}$ and $|\lambda_{a, b}''(t)| \leq C (b-a)^{-2}$ for $t \in \er$.
	\end{enumerate}
	
	We denote 
	$$
	\lambda_i(t) = \lambda_{\tfrac{3}{8}a_{i-1}-\tfrac{1}{4}a_i, \tfrac{7}{16} a_{i-1}-\tfrac{1}{8}{a_i}} (t)
	$$
	and we have (see \eqref{defab})
	\begin{equation}\label{Intemediary}
	|\lambda_i'(t)|\leq C(n,d)a_i^{-1}    \ \text{and} \ |\lambda_i''(t)|\leq C(n,d)a_i^{-2}.
	\end{equation}
	We modify the map $g^*_i$ as follows
	\eqn{defgi}
	$$
		g_i(x) =\bigl[\bar{x}, \bar h_i(\bar{x},x^n)\bigr],
	$$
	where $\bar h_i$ acts from $Q(0, \frac{a_{i-1}}{2}) \subset \mathbb{R}^n$ to $(-\frac{b_{i-1}}{2}, \frac{b_{i-1}}{2}) \subset \mathbb{R}$ as 
	\eqn{defhi}
	$$
	\bar h_i(\bar{x},x^n) := \sgn(x^n)h_i(|x^n|)\prod\limits_{j=1}^{n-1}\lambda_i(|x^j|) + \bigl(1-\prod\limits_{j=1}^{n-1}\lambda_i(|x^j|)\bigr)\frac{b_{i-1}}{a_{i-1}}x^n. 
	$$
	Since $\bar h_i(\bar{x}, t)$ is the convex combination of a pair of increasing functions in $t$ it is clear that $\bar h_i(\bar{x}, t)$ is an increasing function of $t$ for every fixed $\bar{x}$. Further $\bar {h}_i \in \C^{\infty}$. Let us assume then that $g_i(x) = g_i(y)$. Then necessarily $\bar{x} = \bar{y}$ and the fact that $\bar{h}_i$ is strictly increasing implies that also $x^n = y^n$. Thus $g_i$ is injective. Further $g_i$ is $\C^{\infty}$ smooth and $g_i\big(Q(0, \frac{a_{i-1}}{2})\big) = R(0, \frac{a_{i-1}}{2}, \frac{b_{i-1}}{2})$. Due to the definition of $\lambda_i(|x^j|)$, it follows, that
	$$
	g_i(x) = \Bigl[\bar{x}, \frac{b_{i-1}}{a_{i-1}}x^n\Bigr]\quad x\in Q(0, \frac{a_{i-1}}{2}) \setminus Q(0, \tfrac{7}{16} a_{i-1}-\tfrac{1}{8}{a_i}).
	$$
	which is the boundary behaviour we wanted.
	
	From \eqref{Intemediary} we readily get the estimates
	\begin{equation}\label{Intemediary2}
	\Big|D \Big(\prod\limits_{j=1}^{n-1}\lambda_i(|x^j|)\Big)\Big|\leq Ca^{-1}_i \text{ and }
	\Big|D^2 \Big(\prod\limits_{j=1}^{n-1}\lambda_i(|x^j|)\Big)\Big|\leq Ca^{-2}_i.
	\end{equation}
	Then, we can estimate $|D^2(g_i(x))|$ as
	\eqn{nove}
	$$
	\begin{aligned}
			|D^2(g_i(x))|\leq 
			 C \max\Bigl\{&\big|D^2h_i(|x^n|)\big|,\; \Big|D \Big(\prod\limits_{j=1}^{n-1}\lambda_i(|x^j|)\Big) D h_i(|x^n|)\Big|, \Big|h_i(|x_n|) D^2\Big(\prod\limits_{j=1}^{n-1}\lambda_i(|x^j|)\Big)\Big|,\\ 
			&\Big|\frac{b_{i-1}}{a_{i-1}}x^n D^2\Big(\prod\limits_{j=1}^{n-1}\lambda_i(|x^j|)\Big)\Big|, \;\Big|\frac{b_{i-1}}{a_{i-1}}D\Big(\prod\limits_{j=1}^{n-1}\lambda_i(|x^j|)\Big)\Big| \Bigr\}. 
		\end{aligned}
	$$
	With the help of \eqref{hdif1}, \eqref{hdif2} and \eqref{Intemediary2}, we estimate each of the above terms by $Cb_i a_i^{-2}$ and using \eqref{defab} we get
	\begin{equation}\label{gdif}
	|D^2(g_i(x))|\leq C 2^{(\frac{n}{d}-\beta)i}\quad \text{ for all }  x\in Q\bigl(0, \frac{a_{i-1}}{2}\bigr).  
	\end{equation}
	
{\underline{Step 4: Construction of $f$:}}	We build a sequence of homeomorphisms $f_i : [-1,1]^n \to [-1,1]^n$. Let us set $f_0 = x$ and define 
	$$
	f_i(x)=
	\begin{cases}
	f_{i-1}(x)& \quad x \in [-1,1]^n \setminus \bigcup_{v\in \mathbb{V}^i} Q'_v,\\
	f_{i-1}(z_v) + g_i(x-z_v)& \quad x \in Q'_v, v \in \mathbb{V}^i.
	\end{cases}
	$$
	By the definition, we change each $f_i$ only inside the cubes $Q'_v, v \in \mathbb{V}^i$ and $f_i$ coincides with $f_{i-1}$ on a neighborhood of the boundary of $\bigcup_{v\in \mathbb{V}^i} Q'_v$. \seb{Hence} $f_i$ is smooth on $(-1,1)^n$. 
	Clearly $f_i(x)=f_{i-1}(x)=\overline{x}$ in the first $(n-1)$-coordinates and hence they differ only in the $n$-th coordinate $(f_i)_n$. 
	Hence we can use \eqref{hdif1}, \eqref{Intemediary2} and computations similar to \eqref{nove} to obtain 
	$$
	\|Df_i - Df_{i-1}\|_{\infty}\leq \|D(f_i)_n\|_{L^\infty(\bigcup_{v\in \mathbb{V}^i} Q'_v)}+\|D(f_{i-1})_n\|_{L^\infty(\bigcup_{v\in \mathbb{V}^i} Q'_v)}\leq C 2^{-i\beta}. 
	$$
	This means that when we define
	$$
	f(x) = \lim\limits_{i\to\infty} f_i(x)
	$$
	we have that $f_i \to f$ in $\C^1$. As above we have $|D (f_i)_n(x)| \leq C2^{-i\beta}$ on $\C_Q$ for every $i$ and therefore
	$$J_f = 0 \text{ on }\C_Q \ \text{ and } \mathcal{H}^d(\C_Q) > 0.$$
	It is not hard to verify that $f$ is a homeomorphism as it is continuous and one-to-one. 

	We calculate $\|D^2f\|^q_q$ and $\|J^{-a}_f\|_1$. As $f_0$ is the identity, we estimate $\|D^2f\|_q$ as the sum
	$$
	\|D^2f\|_q \leq \|D^2f_0\|_q + \sum\limits_{i=1}\limits^\infty \|D^2f_i - D^2f_{i-1}\|_q. 
	$$
	By the construction, the set, where $f_i \ne f_{i-1}$ can be covered by $2^{ni}$ cubes $Q'_v, v \in \mathbb{V}^i$ and a measure of each cube is by \eqref{defab}
	$$
	|Q'_v| = (2a_i)^n =2^n 2^{-\frac{n^2}{d}i}.
	$$
	For $y \in Q'_v$ we have $f_i(y) = f_{i-1}(z_v)+g_i(y-z_v)$ 
	and we denote $x=y-z_v \in Q(0, \frac{a_{i-1}}{2})$. Since $f_{i-1}$ is an affine mapping inside $Q'_v$ and its second order derivatives are $0$ we have 
	$$
	|D^2f_i(y)-D^2f_{i-1}(y)| = |D^2f_i(x+z_v)| = |D^2g_i(x)|.
	$$
	Together with \eqref{gdif} we obtain
	$$
		\begin{aligned}
			\|D^2f\|^q_q &\leq \sum_{i=1}^\infty \Bigl(\sum_{v\in\mathbb{V}^i} \int_{Q'_v} |D^2f_i(y) - D^2f_{i-1}(y)|^q \, dy \Bigr)^{\frac{1}{q}}\\
			&\leq \sum\limits_{i=1}\limits^\infty \Bigl(2^{ni}\int_{Q'_v} |D^2g_i(x)|^q \, dx \Bigr)^{\frac{1}{q}}\leq C \sum\limits_{i=1}\limits^\infty \Bigl(2^{ni} |Q'_v| 2^{q(\frac{n}{d}-\beta)i}\Bigr)^{\frac{1}{q}}\\
			& \leq C \sum\limits_{i=1}\limits^\infty 2^{(\frac{n}{d}-\beta)i-\frac{n^2}{qd}i+\frac{ni}{q}}.
		\end{aligned}
	$$
	Now we recall that \eqref{FixBeta} gives 
	$$
	(\frac{n}{d}-\beta)-\frac{n^2}{qd}+\frac{n}{q}  < 0
	$$
	and therefore $f\in W^{2,q}((-1,1)^n, \rn)$.
	
	Next we calculate $J_f$ on the set $Q'_v\setminus Q_v$, for $v \in \mathbb{V}^i$. As before, for every $y \in Q'_v\setminus Q_v$ denote $x=y-z_v \in Q(0, \frac{a_{i-1}}{2})$ and so $J_{f}(y) = J_{g_i}(x)$. By \eqref{defgi} we have 
	$$
	J_{g_i}(x) =\frac{\partial(g_i)_n}{\partial x_n} (x) = \frac{\partial \bar {h}_i}{\partial x_n}(x)
	$$ 
	and using \eqref{defhi}, \eqref{hdif1} and \eqref{defab}
	$$
	\frac{\partial \bar {h}_i}{\partial x_n}(\bar{x}, x^n) = \prod\limits_{j=1}^{n-1}\lambda(|x^j|)h'_i(|x^n|) + (1-\prod\limits_{j=1}^{n-1}\lambda(|x^j|))\frac{b_{i-1}}{a_{i-1}} \approx C\frac{b_{i-1}}{a_{i-1}} = C2^{-i\beta}.
	$$
	Then
	$$
		\begin{aligned}
			\|J^{-a}_f\| &\leq C\sum_{i=1}^\infty 2^{ni} |Q'_v| 2^{a\beta i} \leq C \sum_{i=1}^\infty 2^{a\beta i-\frac{n^2}{d}i+ni}.
		\end{aligned}
	$$
	Again using \eqref{FixBeta} we get
	$$
	a\beta -\frac{n^2}{d}+n < 0
	$$
	and therefore $J_f^{-a} \in L^1$. 
\end{proof}

\section{Other counterexamples - Folding, Ball's example and dense $\{J_f=0\}$}

First we give two examples showing that $f\in W^{2,q}$ with $|J_f|^{-a}\in L^1$ might not be a homeomorphism for some values of $q$ and $a$. There seems to be a gap between these counterexamples and positive statement in Theorem \ref{main}. 

\subsection{Folding counterexample to $\haus^{n-1}(\{J_f=0\})=0$ and $J_f\geq 0$ a.e.}

From Theorem \ref{technical} and Corollary \ref{cor} we know that for $q>n$ and $(1-\frac{1}{q})a\geq 1$ we obtain that $\haus^{n-1}(\{J_f=0\})=0$ and thus $J_f$ does not change sign. The following example show the sharpness of this condition.  

\prt{Example}
\begin{proclaim}
Let $n\geq 2$, $q\geq 1$ and $(1-\frac{1}{q})a< 1$. Then there is a continuous mapping $f\in W^{2,q}([-1,1]^n,[-1,1]^n)$ with continuous $Df$ such that $f(x)=x$ on $\partial[-1,1]^n$ and  $|J_f|^{-a}\in L^1((-1,1)^n)$ but 
$$
\{0\}\times(-1,1)^{n-1}\subset\{J_f=0\}\text{ and both } \{J_f>0\}\text{ and }\{J_f<0\}\text{ have positive measure}. 
$$
Moreover, this $f$ is not a homeomorphism. 
\end{proclaim}
\begin{proof}
Since $(1-\frac{1}{q})a< 1$ we can choose $\alpha \in (2-\frac{1}{q}, 1 + \frac{1}{a})$. 
As a first step consider the mapping
$$
g(x)=g(x_1, \dots, x_n) = [|x_1|^\alpha\sgn(x_1), x_2, \dots, x_n].
$$
It is not hard to see, that $J_g(x) =\alpha |x_1|^{\alpha-1}$ and hence $\{x_1=0\}\subset\{J_f=0\}$ and $\{J_g=0\}$ has big $(n-1)$-dimensional projection. At the same time
$$
\alpha<1 + \frac{1}{a}\text{ implies }(\alpha -1)a < 1\text{ and hence }J^{-a}_g \in L^1
$$
and 
$$
\alpha > 2-\frac{1}{q}\text{ implies }(2-\alpha)q<1\text{ and hence }D^2g\in L^q.
$$

It remains to fold this mapping. Let us denote $\x=[x_2,\hdots,x_{n}]$. It is clear that the function
$$
h(\tilde{x})=\min\{0,-\tfrac{1}{2}(1-\|\x\|^2)^3\}
$$
is $C^2$, $-\tfrac{1}{2}\leq h(\tilde{x})\leq 0$ and it is supported in $\er\times [-1,1]^{n-1}$. Since $\alpha\geq 1$ it is also easy to check that $|h(x)|^{\alpha}\in C^2$. 
We want to construct a mapping such that
\eqn{negjac}
$$
A:=\bigl\{x\in[-1,1]^n:\ h(\tilde{x})<x_1<0\bigr\}=\{J_f<0\}.
$$ 
$$
f(x)=\begin{cases}
[x_1^\alpha, x_2, \dots, x_n]&\text{ for }x_1\geq 0,\\
[2^{\alpha-1}|x_1|^\alpha, x_2, \dots, x_n]&\text{ for }\tfrac{1}{2} h(\tilde{x})<x_1<0.\\
[-2^{\alpha-1}(x_1-h(\tilde{x}))^\alpha+|h(\tilde{x})|^{\alpha}, x_2, \dots, x_n]&\text{ for }h(\tilde{x})\leq x_1\leq \tfrac{1}{2} h(\tilde{x}),\\
\bigl[\frac{-1-|h(\tilde{x})|^{\alpha}}{(h(\tilde{x})+1)^{\alpha}}(h(\tilde{x})-x_1)^\alpha+|h(\tilde{x})|^{\alpha}, x_2, \dots, x_n\bigr]&\text{ for }x_1\leq h(\tilde{x}).\\
\end{cases}
$$
Clearly for $x_1=0$, $x_1=\tfrac{1}{2}h(\tilde{x})$ and $x_1=h(\tilde{x})$ the formulas agree so our mapping is continuous. It is also easy to see that for $x_1=-1$ we obtain $f(x)=x$ so our mapping is identity on the boundary. 
Further 
$$
\text{we map }\{h(\tilde{x})<x_1<0\}\text{ onto }\{0<x_1<|h(\tilde{x})|^{\alpha}\}
$$ 
and an easy computation shows that $J_f=\frac{\partial f_1}{\partial x_1}<0$ there. Moreover, it is easy to see that $f$ is not a homeomorphism as each point in $\{0<x_1<|h(\tilde{x})|^{\alpha}\}$ has three preimages. 
It is also not difficult to see that $\{J_f=0\}=\{x_1=0\}\cup\{x_1=h(\tilde{x})\}$. An elementary computation also shows that $Df$ is continuous at $\{x_1=\tfrac{1}{2}h(\tilde{x})\}$ and at $\{x_1=h(\tilde{x})\}$ (note that $|x_1-h(\tilde{x})|^{\alpha-1}=0$ there as $\alpha>1$). 

We know that $h(\tilde{x})$ and $|h(\tilde{x})|$ are $C^2$ and $h(\tilde{x})+1$ is bounded away from $0$. To obtain that $f\in W^{2,q}$ it is thus enough to show integrability of
$$
|x_1|^{(\alpha-2)q}\text{ and }|x_1-h(\tilde{x})|^{(\alpha-2)q}
$$
and this is analogous to integrability of $D^2g$ in the first paragraph using $(2-\alpha)q<1$. Clearly $J_f=\frac{\partial f_1}{\partial x_1}$ so we need to check integrability of 
$$
|x_1|^{(1-\alpha)a}\text{ and }|x_1-h(\tilde{x})|^{(1-\alpha)a}
$$
to obtain $J_f^{-a}\in L^1$ and this follows using $(\alpha -1)a < 1$.
\end{proof}

\subsection{Ball's type example - mapping of a segment to a point} It was known already to J. Ball \cite{Ball} that there are Sobolev mappings with other integrable ``terms" that are equal to a homeomorphism on the boundary but are not a homeomorphism as they map a segment to a point. We give here a generalization of this mapping under our assumptions $f\in W^{2,q}$ and $|J_f|^{-a}\in L^1$.

\prt{Example}
\begin{proclaim}
Let $n\geq 2$, $q\geq 1$ and let $a(\frac{2}{n-1}-\frac{1}{q})<1$. Then there is a continuous mapping $f\in W^{2,q}([-1,1]^{n-1}\times[-2,2],\rn)$ such that 
$f$ is a homeomorphism on $\partial ([-1,1]^{n-1}\times[-2,2])$, $|J_f|^{-a}\in L^1([-1,1]^{n-1}\times[-2,2])$, but $f$ is not a homeomorphism as $f(\{0\}^{n-1} \times [-1,1])=\{0\}^n$. 
\end{proclaim}		

\begin{proof}
Since $a(\frac{2}{n-1}-\frac{1}{q})<1$ we can choose $\beta>0$ such that 
\eqn{defbeta}
$$
\frac{n-1}{a}>\beta > 2 - \frac{n-1}{q}. 
$$
Let us denote $|\bar{x}| =\sqrt{x_1^2 + \dots + x_{n-1}^2}$ and define
$f: [-1,1]^n \to \mathbb{R}^n$ as
$$
f(x) = f(x_1,\dots, x_n) = [x_1, \dots, x_{n-1}, |\bar{x}|^\beta x_n].
$$
Then, $f$ is not a homeomorphism since $f(\{0\}^{n-1} \times [-1,1])=[0,\hdots,0]$. A direct computation gives us 
$$
J_f(x) = |\bar{x}|^\beta\text{ and } |D^2f(x)| \leq C|\bar{x}|^{\beta-2}. 
$$
Now \eqref{defbeta} clearly implies that $f \in W^{2,q}((-1,1)^n, \mathbb{R}^n)$ and $|J^{-a}_f| \in L^1((-1,1)^n)$. 

Now we consider the extension of this mapping to the mapping $\tilde f: [-1,1]^{n-1} \times [-2,2] \to \mathbb{R}^n$ of the form
$$
\tilde f(x) = 
\begin{cases}
    f(x), & \quad x \in [-1,1]^{n},\\
    \bigl[x_1, \dots, x_{n-1}, \sgn(x_n)((|x_n|-1)^2+|\bar{x}|^\beta|x_n|)\bigr], & \quad |x_n|>1.\\
\end{cases}
$$
It is not difficult to see that this mapping is a homeomorphism on $\partial ([-1,1]^{n-1}\times[-2,2])$. It's also easy to see that $|D^2f|$ is bounded on $(-1,1)^{n-1} \times ((-2,2)\setminus[-1,1])$ and thus $\seb{f\in}W^{2,q}((-1,1)^{n-1} \times (-2,2))$. For this mapping we have 
$$
J_{\tilde f}(x) = 2(|x_n|-1) + |\bar{x}|^\beta\geq |\bar{x}|^\beta\text{ for }|x_n|>1
$$ 
and hence $J^{-a}_f \in L^1((-1,1)^{n-1} \times (-2,2))$. 
\end{proof}

\subsection{Example with dense set of $\{J_f=0\}$} 

We know that for $q>n$ the derivative $Df$ of $W^{2,q}$ mapping is continuous and thus the set is $\{J_f=0\}$ is closed. In our Theorem \ref{technical} we show that even for $q\leq n$ the set $\{J_f=0\}$ has small Hausdorff dimension. However, it is good to know that the set $\{J_f=0\}$ might be dense (even though it has small Hausdorff dimension) even for $W^{2,n}$ mapping with $J_f^{-a}\in L^{1}$ for arbitrarily large $a$. For simplicity we construct this only for $n=2$ but this can be easily done in any dimension $n\geq 2$. 

\prt{Example}
\begin{proclaim} Let $n=2$. 
		There exists a homeomorphism $f\in W^{2,2}(B(0,1),\er^2)$ with $J_f^{-a}\in L^{1}(B(0,1))$ for every $a>0$, $f(x) =x$ on $\partial B(0,1)$ such that the set $\{J_f = 0\}$ is dense in $B(0,1)$. 
\end{proclaim}		
\begin{proof}	
	In order to construct our example we will need a family of mappings $f_{x,r}$ for $0<r<\frac{1}{10}$ and $x\in B(0,1)$. We define the preliminary functions
	\eqn{defht}
	$$
		\tilde{h}_{3r}(t) =\begin{cases}
		\frac{t\log r^{-1}}{\log t^{-1}}\quad&\text{ for }0<t\leq r,\\
		(1+\frac{1}{\log r^{-1}})t - \frac{r}{\log r^{-1}}\quad&\text{ for }r\leq t\leq \frac{3}{2}r,\\
		(1-\frac{1}{\log r^{-1}})t + \frac{2r}{\log r^{-1}} \quad&\text{ for } \frac{3}{2}r\leq t\leq 2r,\\
		t \quad&\text{ for } t\geq 2r\\
		\end{cases}
	$$
	and it is easy to check that $\tilde{h}_{3r}(t)$ is continuous and that its derivative is continuous also at $t=r$. We consider the usual convolution kernel $\varphi_r$ (see Step 2 of proof in Section \ref{ahoj}) and we smoothen this function as 	
	$$
		h_{3r}(t) = \begin{cases}
		\tilde{h}_{3r}(t)\quad&\text{ for }t\leq \tfrac{5}{4}r,\\
		\int^{r/4}_{-r/4}\tilde{h}_{3r}(t+s)\varphi_{\frac{r}{4}}(s)\,ds\quad&\text{ for } t\geq \tfrac{5}{4}r.\\
		\end{cases}
	$$
	This function $h_{3r}(t)$ is clearly $C^1$ and belongs to $W_{\loc}^{2,\infty}$ for $t>0$.  Moreover, $h_{3r}(t) = \frac{t\log r^{-1}}{\log t^{-1}}$ on $[0,r)$ and $h_{3r}(t) = t$ for $t\geq 3r$. We define
	$$
		f_{x,r}(y) = \begin{cases}
			x \quad &y=x\\
			x+ h_r(|y-x|)\frac{y-x}{|y-x|} \quad &y\neq x
		\end{cases}
	$$
	and so $f_{x,r} \in \C(B(0,1), \er^2)\cap\C^1(B(0,1)\setminus\{x\}, \er^2)\cap W_{\loc}^{2,\infty}(B(0,1)\setminus\{x\}, \er^2)$ with
	\begin{equation}\label{EqIdentity}
		f_{x,r}(y) = y \quad \text{for} \  \ y\in \er^2\setminus B(x,r).
	\end{equation}
	
	We claim that for the family $f_{x,r}$ the following hold, 
	\begin{equation}\label{DfLinfty}
		\|Df_{x,r}\|_{\infty, \er^2} \leq 2 \qquad \text{for all } x\in \er^2, 0<r<1,
	\end{equation}
	\begin{equation}\label{EqD2o1}
		\lim_{r\to 0}\int_{B(x,r)}|D^2 f_{x,r}|^2 = 0,
	\end{equation}
	\begin{equation}\label{EqJao1}
		\lim_{r\to 0}\int_{B(x,r)}\exp \big(\lambda J_{f_{x,r}}^{-1/3}\big) = 0 \ \text{for all }\lambda >0, \qquad J_{f_{x,r}}(y) >0 \ \ \text{ for all } y\neq x. 
	\end{equation}
	To prove \eqref{DfLinfty} it is enough to consider that using \eqref{der}
	$$
		|Df_{x,r}(y+x)| = |Df_{0,r}(y)| = \max\Big\{ h_{r}'(|y|), \frac{h_r(|y|)}{|y|} \Big\}
	$$
	and calculating $h_r'(t) \leq \|\tilde{h}_r' \|_{\infty}\leq (1+\frac{1}{\log3 r^{-1}})\leq 2$.
	
	By Lemma \ref{radial} we obtain 
	$$
	|D^2f_{x,r}(y+x)| = |D^2f_{0,r}(y)| \approx \max\Bigl\{h''_r(|y|),
	\Bigl|\frac{h'_r(|y|)}{|y|}-\frac{h_r(|y|)}{|y|^2}\Bigr|\Bigr\}. 
	$$
	Clearly 
	$$
		|h''_r(|y|)|\approx\frac{|\log 3r^{-1}|}{|y|\log^2(|y|^{-1})}
		\text{ and }
		\Bigl|\frac{h'_r(|y|)}{|y|} - \frac{h_r(|y|)}{|y|^2}\Bigr|\approx \frac{-\log 3r^{-1}}{|y|\log^2(|y|)}
		\qquad \text{for } |y|\leq \tfrac{1}{3}r.
	$$
For $|y|>\tfrac{1}{3}r$ we can easily estimate
$$
\begin{aligned}
\Bigl|\frac{\tilde{h}'_r(|y|)}{|y|}-\frac{1}{|y|}\Bigr|+
\Bigl|\frac{\tilde{h}_r(|y|)}{|y|^2}-\frac{1}{|y|}\Bigr|
\leq C\frac{1}{r\log r^{-1}}
\leq C \frac{|\log r^{-1}|}{|y|\log^2(|y|^{-1})}
\end{aligned}
$$
and therefore it is not difficult (using e.g. that convolution of identity is identity and convolution of constant is constant) to deduce that for $|y|>\frac{r}{3}$ 
$$
\Bigl|\frac{{h}'_r(|y|)}{|y|}-\frac{{h}_r(|y|)}{|y|^2}\Bigr|\leq C \frac{|\log r^{-1}|}{|y|\log^2(|y|^{-1})}. 
$$
To estimate $|h''_r(|y|)|$ for $|y|>\tfrac{1}{3}r$ note that $\tilde{h}''_r(t)=0$ for $t\in (\frac{r}{3},\frac{r}{2})$, for $t\in (\frac{r}{2},\frac{2r}{3})$ and for $t>\frac{2r}{3}$ (see \eqref{defht}). 
It follows that the derivative of convolution $|h''_r(|y|)|$ can be estimated as the size of the jump of $\tilde{h}'_r$ at $t=\frac{r}{2}$ or $t=\frac{2r}{3}$ divided by the length of the convolution kernel, i.e. 
$$
|h''_r(|y|)|\leq C \frac{1}{\log 3r^{-1}}\frac{12}{r}\approx \frac{|\log r^{-1}|}{|y|\log^2(|y|^{-1})}. 
$$
	Now integrating in polar coordinates 
	$$
		\int_{B(x,r)}|D^2f_{x,r}|^2 \ d\mathcal{L}_2 \leq C\int_0^r \frac{\log^2 r^{-1}}{t^2 \log^4(t^{-1})} t\ dt\leq C\frac{\log^2 r^{-1}}{\log^3 r^{-1}} \xrightarrow{r\to 0} 0.
	$$
	
	To prove \eqref{EqJao1} we have using \eqref{der}
	$$
		J_{f_{x,r}}(y+x) = J_{f_{0,r}}(y) = \frac{h'_r(|y|)h(|y|)}{|y|} 
		\geq C \frac{\log^2 (3r^{-1})}{\log^2 (|y|^{-1})}
		\geq C \frac{1}{\log^2 (|y|^{-1})}
	$$
	and so for any $\lambda >0$ we have
	$$
		\int_{B(x,r)}\exp\big(\lambda J_{f_{x,r}}^{-1/3}\big) \leq C \int_0^r t\exp\Big(\lambda \sqrt[3]{\log^2 t^{-1}}\Big) \ dt \xrightarrow{r \to 0} 0
	$$
	since even  $t\exp\Big(\lambda \sqrt[3]{\log^2 t^{-1}}\Big) \xrightarrow{t\to 0 +}0$.
	
Let us choose a dense countable set of distinct points $\{x_i\}_{i=1}^{\infty}\subset B(0,1)$. We want to construct a mapping such that $J_f(x_i)=0$ for all $i$. 
By induction we construct a set of radii $r_i$ such that $B(x_i,r_i)\subset B(0,1)$ where we choose $r_i$ small enough in each step. We assume that we have already chosen $r_1,\hdots,r_i$  and we want to choose $r_{i+1}$. We fix $\r_{i+1}$ such that 
\begin{equation}\label{DensityPoint}
0<\r_{i+1}<\min\{1-|x_{i+1}|, 4^{-i}|x_{i+1}-x_1|^2,\hdots, 4^{-i}|x_{i+1}-x_i|^2\}
\end{equation}
and we want to choose $0<r<\r_{i+1}$ small enough. 
We define $g_0(x) = x$ and then 
	$$
		g_i = f_{x_{i}, r_{i}}\circ g_{i-1}=f_{x_i,r_i}\circ \dots \circ f_{x_2,r_2}\circ f_{x_1,r_1} \quad \text{and}\quad g_{i+1,r} = f_{x_{i+1}, r}\circ g_i.
	$$
	We have using \eqref{EqIdentity}
	\eqn{odhad}
	$$
		\int_{B(0,1)}\big||D^2g_{i}(x)| - |D^2g_{i+1,r}(x)|\big|^2 \ dx \leq \int_{g^{-1}(B(x_{i+1}, r))} \bigl(3 |D^2g_{i}(x)|^2 + 3|D^2g_{i+1,r}(x)|^2\bigr) \ dx.
	$$
	We estimate the last term using chain rule and \eqref{DfLinfty} as 
	$$
	\begin{aligned}
	\bigl|D^2g_{i+1,r}(x)\bigr|&=\bigl|D^2 f_{x_{i+1}, r}\circ g_i(x)\bigr|=\Bigl|D\bigl(Df_{x_{i+1},r}(g_i(x))Dg_i(x)\bigr)\Bigr|\\
	&\leq |D^2 f_{x_{i+1}, r}(g_i(x))|\cdot |Dg_i(x)|\cdot |Dg_i(x)|+|Df_{x_{i+1}, r}(g_i(x))|\cdot |D^2g_i(x)|\\
	&\leq |D^2 f_{x_{i+1}, r}(g_i(x))|\cdot 2^{2i}+2\cdot |D^2g_i(x)|.\\
	\end{aligned}
	$$
	We set $y = g_{i}^{-1}(x_{i+1})$ and we assume that $r$ is so small that $D^2g_i\approx D^2g_i(y)$ and $J_{g_i}\approx J_{g_i}(y)$ on the whole ball $B(x_{i+1},r)$.  
	By \eqref{odhad} we thus obtain (for our small enough $r$) 
	$$
		\begin{aligned}
			\int_{B(0,1)}\big||D^2g_{i}(x)| - |D^2g_{i+1,r}(x)|\big|^2 \ dx \leq&Cr^2|D^2g_i(y)|^2+\\
			+ C\big(2^{2i} + |D^2g_i(y)|^2\big) &\int_{g^{-1}_i(B(x_{i+1}, r))}(|D^2f_{x_{i+1},r}(g_i(x))|^2 + 4) \ dx.
		\end{aligned}
	$$
	Using \eqref{EqJao1} (specifically that $J_{g_i}(y)>0$) and finally \eqref{EqD2o1} we get
	\begin{equation}\label{DGood}
		\begin{aligned}
			\int_{B(0,1)}\big||D^2g_{i}(x)| -& |D^2g_{i+1,r}(x)|\big|^2 \ dx \\
			&\leq C_{i,y}\Big( r^2 + \int_{B(x_{i+1}, r)}|D^2f_{x_{i+1}, r}(z)|J_{g_i}^{-1}(y)   \ dz \Big) \xrightarrow{r \to 0}0.
		\end{aligned}
	\end{equation}
	Using the same arguments (and finally \eqref{EqJao1} in place of \eqref{EqD2o1}),
	\begin{equation}\label{JGood}
			\begin{aligned}
			\int_{B(0,1)}\big|\exp\big(J_{g_{i}}^{-1/3}(x)\big) &- \exp\big(J_{g_{i+1,r}}^{-1/3}(x)\big)\big|\; dx\leq  \\
			&\leq \int_{g^{-1}(B(x_{i+1}, r))}\bigl(|\exp\big(J_{g_{i}}^{-1/3}(x)\big)| + |\exp\big(J_{g_{i+1,r}}^{-1/3}(x)\big)|\bigr) \ dx\\
			&\leq C_{i,y}r^2+ \int_{g^{-1}_i(B(x_{i+1}, r))} \exp\big(CJ_{g_{i}}^{-1/3}(y)J_{f_{x_{i+1,r}}}^{-1/3}(x)\big) \ dx \\
			&\leq C_{i,y}\Big( r^2 + \int_{B(x_{i+1}, r)}\exp\big(\lambda_{i,y} J_{f_{x_{i+1,r}}}^{-1/3}(z)\big)   \ dz \Big) \xrightarrow{r \to 0} 0.
		\end{aligned}
	\end{equation}
	
We use \eqref{DGood} and \eqref{JGood} to find $0< r_{i+1}=r < \tilde{r}_{i+1}$ small enough such that 
	$$
		\int_{B(0,1)}\big||D^2g_{i+1}| - |D^2g_{i}|\big|^2 \leq 2^{-2i}
	$$
	and
	$$
		\int_{B(0,1)}\big|\exp\big(J_{g_{i}}^{-1/3}(x)\big)- \exp\big(J_{g_{i+1}}^{-1/3}(x)\big)\big| \ dx <2^{-i}. 
	$$
	We call $f(x) = \lim_{i\to \infty}g_i(x)$. Then $g_i$ is a Cauchy sequence in $W^{2,2}$ and $f\in W^{2,2}$ with 
	$$
		\|D^2 f\|_2 \leq \sum_{i=0}^\infty \|D^2g_{i+1} - D^2g_{i}\|_2 \leq 2.
	$$
	 This (and the boundedness of $f$) immediately gives us the uniform convergence of $g_i$ on $B(0,1)$. It is obvious that $f$, as the composition of injective maps, is injective and hence a homeomorphism with $f(x)=x$ on $\partial B(0,1)$. 
	 
	 Further we have that $\exp\big(J_{g_{i}}^{-1/3}(x)\big)$ is a Cauchy sequence in $L^1$ (and since $J_{g_i}(x) = 0$ exactly if $x$ equals some $x_i$) we have $J_{g_i}(x) \to J_f(x)$ almost everywhere. Moreover, 
	$$
			\int_{B(0,1)}\exp\big(J_{f}^{-1/3}(x)\big) <\infty 
	$$
	implies that $\int_{B(0,1)} J_f^{-a}<\infty$ for any $a>0$. 
	
It is not difficult to find out that $Dg_k(x_k)=0$ and hence $J_{g_k}(x_k)=0$. We have
	$$
		\begin{aligned}
			\oint_{B(x_k,r)}|J_f(x)-J_f(x_k)|\; dx &\leq Cr^{-2} \int_{B(x_k,r)} J_{g_k}(x) \; dx+ Cr^{-2} \sum_{x_i\in B(x_k, r)} \int_{B(x_i,r_i)} J_{g_i}(x) \; dx.
		\end{aligned}
	$$
	Since $|Dg_k(x)| \leq \frac{C_k}{\log t^{-1}}$ on $B(x_k, r)\setminus\{x_k\}$ we have that the first term tends to zero as $r\to 0^+$. For the second term we use the condition $r_i < \min\{4^{-i}|x_i - x_j|^2: 1\leq j<i\}$  (see \eqref{DensityPoint}) which implies that $i\geq k$ for every $x_i\in B(x_k, r)$ and $|Dg_i|\leq 2^i$, i.e. $J_{g_i} \leq 4^i$ to get
	$$
		\begin{aligned}
			r^{-2} \sum_{x_i\in B(x_k, r)} \int_{B(x_i,r)} J_{g_i}(x) \; dx & \leq Cr^{-2} \sum_{x_i\in B(x_k, r)} 4^{i}4^{-2i}r^{4} \leq Cr^2
		\end{aligned}
	$$
	and so these terms also tend to zero as $r\to 0^+$. Therefore $J_f(x_k)=0$ is a Lebesgue value and the set $\{J_f = 0\}$ is dense.
	
	It is not difficult to see that by appropriate choice of $\{r_{i}\}$ we can even make $\int_{B(0,1)}|D^2f| + |J_f|^{-a}\; dx$ as close to $|B(0,1)|$ as we like. 
	
\end{proof}



\section{Other properties}

In this short section we give two applications where it is useful to know the size of the critical set $\{J_f=0\}$. 

\subsection{Approximation by smooth homeomorphisms}

We show that the usual convolution approximations are also a homeomorphism if we are far away from the critical set $\{J_f=0\}$ for $q>n$ and $f$ satisfying the assumptions of Theorem \ref{main}. 

\begin{prop}\label{Approximation}
	For $n\geq 3$ we assume that $q>n$ and $a>n-1$. For $n=2$ we assume that $q>2$, $a\geq 1$ and $(\frac{3}{2}-\frac{2}{q})a\geq 1$. Let $\Omega \subset \rn$ be a Lipschitz domain, let $f_0:\overline{\Omega} \to \rn$ be a homeomorphism and let $f\in W^{2,q}(\Omega, \rn)$ be such that $f=f_0$ on $\partial\Omega$ and $|J_f|^{-a} \in L^1(\Omega)$. Let $\delta >0$ and let an open $G\subset\overline{G}\subset \Omega$ satisfy $J_f \geq3\delta >0$ on $G$. Then there exists a $k_0\in \en$ such that $f_k = f*\psi_k$, the standard convolution approximation of $f$, is a homeomorphism on $G$ for each $k\geq k_0$. 
\end{prop}
\begin{proof} Since $q>n$ we have that $Df$ and so also $J_f$ are $(1-\frac{n}{q})$-H\"older continuous. Therefore there exists $k_1\in \en$ and a domain $U$ such that 
$$
G\subset G+B(0,k_1^{-1})\subset U\subset U+B(0,k_1^{-1})\subset \Omega\text{  such that }
J_f(x)\geq 2\delta\text{ for all }x\in U.
$$ 
By Theorem~\ref{main} we have that $f$ is a homeomorphism on $\Omega$ and so $\dist (f(\partial U), f(G))=:3\eta>0$. We know that convolution approximations $f_k$ converge to $f$ strongly in $W^{2,q}$ (see e.g. \cite[Chapter 4, Theorem 1 $(v)$ and $(vi)$]{EG}) and hence the Sobolev embedding theorem gives us 
$$
\|f_k-f\|_{L^\infty(U)} \to 0\text{ and }\|J_{f_k}-J_f\|_{L^\infty(U)} \to 0.
$$ 
Therefore we find a $k_0\geq k_1$ such that $\|f_k-f\|_{\infty, U} \leq \eta$ and $J_{f_k}(x)\geq \delta$ for all $k\geq k_0$ and all $x\in U$.
	
	Next, for every $y\in f_k(G)$ we have that $y$ is a regular value for $f_k$ on $G$ i.e.\ for all $x \in f_k^{-1}(y)\cap G$ we have that $Df(x)$ exists and $J_f(x)\neq 0$. Therefore, for $k\geq k_0$, we have that $\haus^0(f^{-1}_k(y)) = \deg(y,f_k,G)$ (as $J_{f_k}>0$ on $G$, see e.g. \cite[proof of Lemma 3.11]{HK14}). Assume that $x\in G$ then for every $z\in f(\partial U)$ we have
	$$
		|f_k(x) - z| \geq |f(x) - z| - |f_k(x) - f(x)| \geq 3\eta - \eta =2\eta
	$$
	Using the fact that $\|f_k-f\|_{\infty, U} \leq \eta$ and the homotopic property of the degree (see e.g. \cite[Theorem 2.3. (1)]{FG}) we have that 
	$$
	\deg(f_k(x),f_k,U) = \deg(f_k(x),f,U) = 1
	$$ 
	because $f$ is a homeomorphism. But then $f^{-1}_k(f_k(x))\cap G= \{x\}$ is a singleton and therefore $f_k$ is injective on $G$.
\end{proof}

\subsection{Euler-Lagrange equation}
It has been shown by Healey and Kr\"omer in \cite{HK}, that the minimization problem of finding
\begin{equation}\label{Problem}
\min \big\{E(f);  \  f\in W^{2,q}(\Omega, \er^n),\ f=f_0\text{ on }\partial\Omega,\  J_f >0 \text{ a.e.} \big\},
\end{equation}
where
\begin{equation}\label{EnergyDef}
	E(f) = \int_{\Omega} \bigl(\Psi(D^2 f(x)) + \Phi(Df(x))\bigr) dx
\end{equation}
has a minimizer under suitable conditions. 
Further, under the condition that
\eqn{GrowDPsi}
$$
\begin{aligned}
	\Phi(A) &\in \C^1\text{ on }\{A:\ \det A>0\},\ \Phi(A)=\infty\text{ if }\det A\leq 0,\\ 
	\Psi &\in \C^1 \quad \text{and} \quad |D\Psi(G)| \leq C(1+|G|)^{q-1},
	\end{aligned}
$$
they proved (using that in their setting $J_f\geq \delta >0$ for all contenders $f$) that the minimizer $f^*$ satisfies the weak Euler-Lagrange equation. We prove that this equations holds also in our situation if we are far away from the critical set $\{J_f=0\}$. 



\begin{prop}\label{EL}
	Let $q > n\geq 2$ and let $E$ be defined by \eqref{EnergyDef} where $\Psi, \Phi$ satisfy 
	\eqref{GrowDPsi}. 
	Assume that there a minimizer of \eqref{Problem} and call it $f^*$. Then $f^*$ satisfies
	\begin{equation}\label{ELeq}
	\int_{\Omega}D\Psi(D^2f^*(x))\cdot D^2\phi(x) + D\Phi(Df^*(x))\cdot D\phi(x) \, dx = 0
	\end{equation}
	for all $\phi \in W^{2,q}_0(\Omega\setminus\{J_{{f^*}}\neq 0\} ,\rn)$ with $D\Phi(Df^*)\cdot D\phi\in L^1(\Omega)$.
\end{prop}

\begin{proof}
As $E(f^*)$ is finite we find $J_{f^*}>0$ a.e.. 
	Let us first consider $\phi \in C^{\infty}_{c}(\Omega\setminus\{J_{f^*}=0\},\rn)$. The compact support of $\phi$ and the continuity of $J_{f^*}$ implies that there is $\delta>0$, such that 
	\[
	\min_{\{x\in \mathrm{supp}(\phi)\}}J_{f^*}=\delta
	\]
	Now let $h >0$. By the fact that $f^*$ is a minimizer we have that
	$$
	\begin{aligned}
	0 &\leq \frac{E(f^* + h\phi) - E(f^*)}{h}\\
	& = \int_{\Omega_{\delta}}\frac{\Psi(D^2f^* + hD^2\phi) - \Psi(D^2f^*) +  \Phi(Df^* + hD\phi) - \Phi(Df^*)}{h} \, dx\\
	& = \int_{\Omega_{\delta}}\int_0^1\frac{1}{h}\frac{\partial}{\partial t}\Big[\Psi(D^2f^* + htD^2\phi) +  \Phi(Df^* + htD\phi) \Big] \, dt \, dx\\
	& = \int_{\Omega_{\delta}} \int_0^1 D\Psi(D^2f^*(x) + htD^2\phi(x))D^2\phi(x) +  D\Phi(Df^*(x) + htD\phi(x))D\phi(x) \, dt \, dx.
	\end{aligned} 
	$$
	We know that $\|D\phi\|$ is bounded and therefore, there exists $h_0>0$ sufficiently small such that 
	$$
	\det\bigl(Df^*(x) +htD\phi(x) \bigr)\geq \tfrac{1}{2}\delta\text{ for all }t\in[0,1], h\in[-h_0,h_0]\text{ and all }x\in \{\phi\neq 0\}.
	$$ 
	Therefore the $D\Phi$ term is bounded irrespective of $t$ and $h$.
	
	Thanks to the fact that $D^2\phi$ is bounded we find $h_1>0$ such that
	$$
	|D^2f^*(x) + htD^2\phi(x)| \leq (1+|D^2f^*(x)|)
	$$
	for all $t\in[0,1]$, $h\in [-h_1,h_1]$ and all $x\in \{\phi\neq 0\}$. Then using \eqref{GrowDPsi} we have
	$$
	\bigl|D\Psi(D^2f^* + htD^2\phi)\cdot D^2\phi \bigr| \leq C(1+|D^2f^*|)^{q-1}
	$$
	which is integrable. Now we let $h\to 0$ and the dominated convergence theorem gives us 
	$$
	0 \leq \int_{\Omega_{\delta}}D\Psi(D^2f^*(x))D^2\phi(x) +  D\Phi(Df^*(x))D\phi(x) \, dx.
	$$
	The same argument, replacing $\phi$ with $-\phi$, shows \eqref{ELeq} for all $\phi\in \C_c^{\infty}(\Omega\setminus\{J_f=0\},\rn)$.
	
	Now assume that $\phi \in W^{2,q}(\Omega,\rn)$ and $\overline{\{\phi\neq 0\}}\cap (\partial\Omega\cup\{J_f=0\})=\emptyset$. We find $\phi_{k} \in \C_c^{\infty}$ such that $\phi_k \to \phi$ in $W^{2,q}(\Omega, \rn)$ and such that there is 
	$$
	G\subset\overline{G}\subset\Omega\text{ with }\dist\bigl(G,\{J_f=0\}\bigr)>0\text{ and }
	\{\varphi\neq 0\}\cup \bigcup_{k=1}^{\infty}\{\varphi_k\neq 0\}\subset G.
	$$
	Using \eqref{GrowDPsi} and the H\"older inequality we have
	$$
		\begin{aligned}
			\bigg|\int_{\Omega}D\Psi(D^2f^*)[D^2\phi -D^2\phi_k] \bigg|  
			&  \leq C\int_{\Omega}(1+|D^2f^*|)^{q-1}|D^2\phi -D^2\phi_k| \\
			& 	\leq C\Big[\int_{\Omega}(1+|D^2f^*|)^{q}\Big]^{\tfrac{q-1}{q}}\Big[\int_{\Omega}|D^2\phi -D^2\phi_k|^{q}\Big]^{1/q}
		\end{aligned}
	$$
	and this tends to zero as $k\to \infty$. Since $q>n$ we have that $\|D\phi -D\phi_k\|_{L^\infty(\Omega)} \to 0$ as $k \to \infty$. Since $\Phi$ is $\C^1$ on $\{A: \det A>0\}$, we have that $\|D\Phi(Df^*)\|_{L^{\infty}(G)} < \infty$. Then
	$$
			\Big|\int_{G}D\Phi(Df^*)[D\phi -D\phi_k] \Big| \leq C\|D\Phi(Df^*)\|_{L^\infty(G)}\int_{G} \big|D\phi -D\phi_k\big|
	$$
	and this tends to zero as $k\to \infty$. Since $G$ is arbritary we get a.e.\ convergence. Further, we may fix some $G$ that $|D\Phi(Df^*)||D\phi_k|$ is monotonously increasing on $\Omega\setminus G$. Then dominated convergence implies
	$$
			0 = \lim_{k\to \infty}\int_{\Omega}D\Psi(D^2f^*)D^2\phi_k + D\Phi(Df^*)D\phi_k\, dx
			 = \int_{\Omega}D\Psi(D^2f^*)D^2\phi +  D\Phi(Df^*)D\phi
	$$
	and we have proved the validity of \eqref{ELeq}.
\end{proof}

\subsection*{Data availability statement}
Data sharing not applicable to this article as no datasets were generated or analysed during the current study.

\end{document}